\RequirePackage{fix-cm}
\documentclass[smallextended,envcountsect,envcountsame]{svjour3hack}   

\usepackage[leqno]{amsmath}
\usepackage{amssymb}
\usepackage{subcaption}
\usepackage[dvipsnames]{xcolor}
\usepackage{enumitem}

\smartqed 
\usepackage{graphicx}

\usepackage[hypertexnames=false,colorlinks=true,breaklinks=true,bookmarks=true,urlcolor=blue,citecolor=blue,linkcolor=blue,bookmarksopen=false,draft=false]{hyperref}

\pdfoutput=1  

\usepackage[noblocks]{authblk}

\title{Convex relaxation for the generalized\\ maximum-entropy sampling problem\footnote{A short preliminary version of this work 
appeared in the proceedings of the Symposium on Experimental Algorithms (SEA) 2024, in the series of Dagstuhl's  Leibniz International Proceedings in Informatics (LIPIcs); see \cite{SEA_proceedings}.
}
}

\titlerunning{Convex relaxation for GMESP} 
\authorrunning{G. Ponte, M. Fampa, J. Lee} 

\author{Gabriel Ponte \and Marcia Fampa \and Jon Lee}

\institute{Gabriel Ponte, ORCID: 0000-0002-8878-6647\at
University of Michigan. \email{gabponte@umich.edu}      
         \and
Marcia Fampa, ORCID: 0000-0002-6254-1510\at
Universidade Federal do Rio de Janeiro. \email{fampa@cos.ufrj.br}
         \and
CORRESPONDING AUTHOR: Jon Lee, ORCID: 0000-0002-8190-1091\at
 University of Michigan. \email{jonxlee@umich.edu} 
}

\date{\small Revised June 25, 2025}

\DeclareMathOperator{\Diag}{Diag}
\DeclareMathOperator{\diag}{diag}

\DeclareMathOperator{\conv}{conv}
\DeclareMathOperator{\Trace}{tr}
\DeclareMathOperator{\rank}{rank}

\newtheorem{thm}{Theorem}{\bf}{\rm}
{\bf}{\rm}
\newtheorem{lem}[thm]{Lemma}{\bf}{\rm}
{\bf}{\rm}
{\bf}{\rm}
\newtheorem{rem}[thm]{Remark}{\bf}{\rm}
{\bf}{\rm}
{\bf}{\rm}

\makeatletter
\let\c@prop\c@theorem
\let\c@cor\c@theorem
\let\c@lemma\c@theorem
\let\c@definition\c@theorem
\let\c@example\c@theorem
\let\c@remark\c@theorem
\let\c@obs\c@theorem
\let\c@claim\c@theorem
\makeatother

\makeatletter
\renewcommand*{\top}{
  {\mathpalette\@transpose{}}
}
\newcommand*{\@transpose}[2]{
  \raisebox{\depth}{$\m@th#1\scriptscriptstyle\mathsf{T}$}
}
\makeatother

\begin{document}

\maketitle

\begin{abstract}
The \emph{generalized maximum-entropy sampling problem} (GMESP) is to select an order-$s$ principal submatrix from an order-$n$ covariance matrix, to maximize the product of its $t$ greatest eigenvalues, $0<t\leq s <n$. 
Introduced more than 25 years ago,
GMESP is a natural generalization of two fundamental problems in statistical design theory:
(i) maximum-entropy sampling problem (MESP);
(ii) binary D-optimality (D-Opt). 
In the general case, it can be
motivated by   
a selection problem in the context of principal component analysis (PCA).

We introduce the first convex-optimization based relaxation for GMESP, study its behavior, compare it to an earlier spectral bound, and demonstrate its use in a branch-and-bound scheme. We find that
such an approach is practical when $s-t$ is very small. 


\end{abstract}


\section{Introduction}\label{sec:int}

Throughout, $C$ is a symmetric positive-semidefinite matrix with rows/columns
indexed from $N:=\{1,2,\ldots,n\}$, with $n >1$ and $r:=\rank(C)$.
For integers $t$ and $s$, such that $0<t\leq r$ and $t\leq s<n$, we define the \emph{generalized maximum-entropy sampling problem} (see \cite{WilliamsPhD,LeeLind2020})
\begin{equation}\tag{GMESP}\label{GMESP}
\begin{array}{ll}
z(C,s,t):=&\max \left\{\sum_{\ell=1}^t \log (\lambda_\ell(C[S(x),S(x)])) ~:~ 
\right.\\
&\qquad\qquad\qquad\qquad\left.\mathbf{e}^\top x =s,~ x\in\{0,1\}^n\vphantom{\sum_{\ell=1}^t }\right\},
\end{array}
\end{equation}
where $S(x):=\{j\in N ~:~ x_j\neq 0\}$ denotes the \emph{support} of $x\in\{0,1\}^n$, 
 $C[S,S]$ (for $\emptyset\neq S\subset N$) denotes the principal submatrix of $C$ indexed by $S$, and
$\lambda_\ell(X)$ denotes the $\ell$-th greatest eigenvalue of a symmetric matrix $X$. 

More than twenty-five years ago, \ref{GMESP} was introduced in the Ph.D. dissertation of Joy Lind\footnote{n\'ee Williams} as a common generalization of the maximum-entropy sampling problem (MESP) and the binary D-optimality problem (D-Opt) (see \cite{WilliamsPhD}, but not widely disseminated until \cite{LeeLind2020}).
 MESP, a central problem in statistics and information theory,  corresponds to the problem of selecting a subvector of size $s$ from
 a Gaussian $n$-vector, so as to maximize the ``differential entropy''
 (see \cite{Shannon}) of the chosen subvector; see \cite{FLBook} for an in-depth treatment. 
MESP is the special case of \ref{GMESP} for which $t:=s$.
The relationship with binary D-Opt is a bit more involved. Given an $n\times r$ matrix $\mathcal{A}$ of full column rank,
binary D-Opt corresponds to the special case of \ref{GMESP} for which $C:=\mathcal{A}\mathcal{A}^\top$, and $t:=r$. Binary D-Opt is equivalent to the problem of selecting a set of $s$ design points from a given set of $n$ potential design points (the rows of $\mathcal{A}$), so as to minimize the volume of a confidence ellipsoid for the least-squares parameter estimates in the resulting linear model (assuming additive Gaussian noise); see \cite{PonteFampaLeeMPB}, for example. 

For the general case of \ref{GMESP}, we can see it as motivated by 
a selection problem in the context of principal component analysis (PCA); see, for example,  \cite{PCA} and the references therein, for the very important topic of PCA.
Specifically, \ref{GMESP} amounts to 
selecting a subvector of size $s$ from
 a Gaussian $n$-vector, so that the geometric mean of the variances associated with the $t$ largest principal components is maximized.
 Linking this back to MESP, we
 can see that problem as selecting a subvector of size $s$ from
 a Gaussian $n$-vector, so that the geometric mean of the variances associated with \emph{all} principal components is maximized. 
Working with the geometric mean of the variances, encourages a selection
where all $t$ of them are large and similar in value.\footnote{In this spirit, the product of sample variances is used in Bartlett's test of homogeneity of variances (for the normal case); see \cite[Section 10.21, pp. 296]{Snedecor}, and available for example in \texttt{SciPy} as \texttt{scipy.stats.bartlett}, \url{https://docs.scipy.org/doc/scipy/reference/generated/scipy.stats.bartlett.html}.} We note that maximizing the geometric mean of some positive numbers is equivalent to maximizing the sum of their logarithms, which links back to our objective function in \ref{GMESP}.
 
 Expanding on our motivation for \ref{GMESP}, we assume that 
 we are in a setting where we have $n$ observable Gaussian random variables, with
 a possibly rank-deficient covariance matrix. We assume that observations are costly, and so we 
 want to select $s \ll n$  for observation. Even the $s$ selected random variables 
 may have a low-rank covariance matrix. 
 Posterior to the selection, we would then carry out
PCA on the associated order-$s$ covariance matrix,
with the aim of identifying the most informative $t < s$ latent/hidden random variables,
where we define \emph{most informative} as corresponding to maximizing the geometric mean of the variances of the 
$t$ dominant principal components. 

We also define the \emph{constrained generalized  maximum-entropy sampling problem}
\begin{equation}\tag{CGMESP}\label{CGMESP} 
\begin{array}{ll}
z(C,s,t,A,b):=&\max \left\{\textstyle \vphantom{\sum_{j\in S}} \sum_{\ell=1}^t \log (\lambda_\ell(C[S(x),S(x)])) ~:~\right.\\
& \qquad\qquad\quad\left. \mathbf{e}^\top x =s,~Ax\leq b,~ x\in\{0,1\}^n\vphantom{\sum_{\ell=1}^t }\right\},
\end{array}
\end{equation}
where $A\in\mathbb{R}^{m\times n}$  and $b\in\mathbb{R}^m$, which is very useful in practical applications where there are budget constraints, logistical constraints, etc., on which sets of size $s$ are feasible. Correspondingly, we also refer to CMESP (the constrained version of MESP) and binary CD-Opt (the constrained version of binary D-Opt). 

The special cases of MESP and binary D-Opt are already NP-hard (see \cite{KLQ}
and \cite[Sec. 5]{gritz}), and the main approach for solving them (and the more general \ref{GMESP} and \ref{CGMESP}) to optimality is B\&B (branch-and-bound) (see \cite{LeeLind2020}). 
Lower bounds are calculated by local search (especially for \ref{GMESP}, in contrast to \ref{CGMESP}), rounding, etc.
Upper bounds are calculated in a variety of ways.
The only upper-bounding method
in the literature for \ref{GMESP}/\ref{CGMESP} uses spectral information; see
\cite{WilliamsPhD,LeeLind2020}.

Some very good  upper-bounding methods for CMESP and CD-Opt
are based on convex relaxations; see 
\cite{AFLW_Using,Anstreicher_BQP_entropy,Kurt_linx,Nikolov,Weijun,ChenFampaLee_Fact,PonteFampaLeeMPB}. 
For CMESP, a  ``down branch''
is realized by deleting a symmetric row/column pair from $C$.
An ``up branch'' corresponds to 
calculating a Schur complement. 
For binary CD-Opt, a ``down branch'' amounts to eliminating a row from $\mathcal{A}$, and 
an ``up branch'' corresponds to 
adding a rank-1 symmetric matrix before applying a determinant operator to an order-$r$  symmetric matrix that is linear in $x$ (see \cite{li2022d} and \cite{PonteFampaLeeMPB} for details).

For the general case of \ref{CGMESP} and the spectral bounding technique,
we refer to \cite{LeeLind2020} for a discussion of an ``up branch'', which is actually quite complicated and probably not very efficient. We will present a convex relaxation
for \ref{CGMESP} that is amenable to the use of a simple ``up branch''.  Our new ``generalized factorization bound''
for \ref{CGMESP} directly generalizes (i) the ``factorization bound''
for CMESP (see \cite{ChenFampaLee_Fact}), and (ii) the ``natural bound'' for 
binary CD-Opt (see \cite{PonteFampaLeeMPB}). 
We wish to emphasize that it does \emph{not directly generalize}
the ``factorization bound'' for binary CD-Opt (also from \cite{PonteFampaLeeMPB}). However, if we recast a binary CD-Opt instance as a CMESP instance, employing the transformation of \cite[Remark 8]{PonteFampaLeeMPB}, and then recast that instance of CMESP as an instance of \ref{CGMESP}, and then apply our new ``generalized factorization bound'' to that instance of \ref{CGMESP}, we obtain the same bound as the 
 ``factorization bound'' for the binary CD-Opt instance that we started with. So, in a sense, our new ``generalized factorization bound'' unifies  the ``natural bound'' and the
  ``factorization bound'' for binary CD-Opt.

\medskip
\noindent {\bf Organization and Contributions.} 
In \S\ref{sec:bounds}, we introduce the ``generalized factorization bound'' for \ref{CGMESP} as the Lagrangian dual of a non-convex relaxation and establish its fundamental properties and its relation with the spectral bound for \ref{CGMESP} from \cite{LeeLind2020}. Although the development of the 
bounding formulation ``DGFact''  closely mirrors the corresponding development for the special case of CMESP (i.e., $t\!:=\!s$), there are 
a lot of details that need to be carefully worked through
to rigorously establish the bound. 
Further, while the proofs (which we omit) of some very-useful basic properties (Theorems \ref{thm:scale}, \ref{cor:FactDoesntMatter} and \ref{thm:fixFact}) 
closely follow the corresponding proofs for CMESP,
there is more to establishing other fundamental results (Theorem \ref{thm:almost_dominates}), and others
have no analog for CMESP
(Theorems \ref{thm:concavity} and \ref{thm:disc}).

In \S\ref{sec:DDFact}, we apply Lagrangian duality again,
reaching a more tractable formulation \ref{DDGFact} for calculating the 
generalized factorization bound.
Again, although the development closely mirrors the corresponding development for the special case of CMESP, there are 
a lot of details that need to be carefully worked through. Additionally, we wish to highlight Theorem \ref{dualoptimallcfa} (and Lemma \ref{lem:optimal_xhat_diagf})
which even for MESP and binary D-Opt is completely new.

In \S\ref{sec:exp}, we present results from computational experiments with a B\&B algorithm based on the generalized factorization bound, where we
demonstrate favorable computational performance, for \ref{GMESP}, when $s-t$ is quite small. For  \ref{CGMESP}, the results are even more promising, with computational times to solve the test instances significantly reduced, compared to \ref{GMESP}. Finally, we present  numerical comparisons between the spectral bound and the generalized factorization bound for both \ref{GMESP} and \ref{CGMESP}, showing much better results for the latter  when $n-s$ is large and $s-t$ is small. It is unfortunate, from the perspective of the motivating PCA selection problem, that we do not yet have a very effective technique for the case in which $s-t$ is large.

In \S\ref{sec:out}, we describe some directions for further study. 

\medskip
\noindent {\bf Notation.}
 We let $\mathbb{S}^n_+$ (resp., $\mathbb{S}^n_{++}$) denote the set of
 positive semidefinite (resp., definite) symmetric matrices of order $n$. We let $\Diag(x)$ denote the $n\times n$ diagonal matrix with diagonal elements given by the components of $x\in \mathbb{R}^n$, and $\diag(X)$ denote the $n$-dimensional vector with elements given by the diagonal elements of $X\in\mathbb{R}^{n\times n}$.
We denote an all-ones  vector
by $\mathbf{e}$ and the $i$-th standard unit vector by $\mathbf{e}_i$\thinspace.
For matrices $A$ and $B$ with the compatible shapes,
$A\bullet B:=\Trace(A^\top B)$ is the matrix dot-product.
For a matrix $A$, we denote row $i$ by $A_{i\cdot}$ and
column $j$ by $A_{\cdot j}$\thinspace.

\section{Generalized factorization bound}\label{sec:bounds}

Suppose that $r:=\rank(C)\geq t$.
We factorize $C=FF^\top$,
with $F\in \mathbb{R}^{n\times k}$, for some $k$ satisfying $r\le k \le n$. 
This could be a 
Cholesky-type factorization, as in  \cite{Nikolov} and \cite{Weijun}, where $F$ is lower triangular and $k:=r$,   it could be derived from
a spectral decomposition  $C=\sum_{i=1}^r \lambda_i v_i v_i^\top$\,, by selecting   $\sqrt{\lambda_i}v_i$ as the column $i$ of $F$,
 $i=1,\ldots,k:=r$, 
 or it could be derived from 
  the matrix square root of $C$, where $F:=C^{1/2}$, and  $k:=n$.

For $x\in [0,1]^n$, we define
$f(x) :=  \sum_{\ell=1}^t \log\left( \lambda_{\ell}  (F(x))\right)$, where
$ F(x) := \sum_{i\in N} F_{i\cdot}^\top F_{i\cdot} ~x_i~ =F^\top \Diag(x)F$, and
\[
z_{{\mbox{\tiny GFact}}}(C,s,t,A,b;F):=\max \left\{f(x) \!~:~\! \mathbf{e}^{\top}x=s,~ Ax\leq b,~0\leq x\leq \mathbf{e} \right\}.\tag{GFact}\label{GFact}
\]
Next, we see that \ref{GFact} gives an upper bound for \ref{CGMESP}.
 This generalizes the corresponding result for CMESP (see \cite[Thm. 3.4.1]{FLBook}).

\begin{thm}\label{thm:Fact_bound}
\[
z(C,s,t,A,b)\leq 	z_{{\mbox{\tiny\rm GFact}}}(C,s,t,A,b;F).
\]
\end{thm}

\begin{proof}
It suffices to show that for any feasible solution $x$ of \ref{CGMESP} with finite objective value,
we have $\sum_{\ell=1}^t \log\left( \lambda_{\ell} (C[S(x),S(x)] )\right) = f(x)$. Let $S:=S(x)$.
Then, for $S\subset N$ with $|S|=s$ and $\rank(C[S,S])\geq t$,
 we have
$
F(x) =  \sum_{i\in S} F_{i\cdot}^\top F_{i\cdot} = F[S,\cdot]^\top F[S,\cdot] \in \mathbb{S}_{+}^k\,.
$
Also, we have
$
C[S,S]=F[S,\cdot]F[S,\cdot]^\top \in \mathbb{S}_{+}^s\,.
$
Now, we observe that the nonzero eigenvalues of the matrices
$F[S,\cdot]^\top F[S,\cdot]$ and $F[S,\cdot]F[S,\cdot]^\top$ 
are identical, and the rank of these matrices
is at least $t$. So, the $t$ largest eigenvalues of
these matrices are positive and identical.  The result follows.
\qed\end{proof}
  
From the proof, we  see that  \ref{GFact} is an \emph{exact relaxation} of \ref{CGMESP},
in the sense that the feasible region of \ref{CGMESP} is contained in the 
feasible region of  \ref{GFact}, and feasible solutions of \ref{CGMESP} have the same objective value in \ref{GFact}  that they do in \ref{CGMESP}.
But \ref{GFact} is \emph{not} generally a convex-optimization problem\footnote{Even for $t=1$, note that while log is of course a concave function,
the maximum eigenvalue (inside the log) is convex in $x$. So the composed function, which we wish to maximize, is not generally a concave function.}, so we
cannot make direct use of such a relaxation in a B\&B based on convex relaxation. Alternatively, one might consider a global-optimization approach that handles nonconvexities at the algorithmic level, using so-called ``spatial B\&B'' (see \cite{Baron1996}, for example). But such an approach should be relatively inefficient, and moreover, there is no clear way to induce separability 
of the objective function, which is a key step in spatial B\&B.

 Closely following a technique in the literature, 
 and aiming for a convex relaxation, 
 we 
will overcome the difficulty 
(of the nonconvexity of the objective function of \ref{GFact})
using  Lagrangian duality,  
obtaining an upper bound for $z_{{\mbox{\tiny GFact}}}$\,. 
We note that although we closely follow a 
known technique, it is important to carefully work out the technical  differences 
when $t<s$.
We first re-cast \ref{GFact} as 
	\[
	\textstyle \max \left\{ \sum_{\ell=1}^t \log \left(\lambda_{\ell} \left( W \right)  \right)
	~:~F(x)=W,~ \mathbf{e}^{\top}x=s,~ Ax\leq b,~0\leq x\leq \mathbf{e} \right\},
	\]
and consider the Lagrangian function
\begin{align*}
&\mathcal{L}(W,x,\Theta,\upsilon,\nu,\pi,\tau):= \textstyle \sum_{\ell=1}^t \log \left(\lambda_{\ell} \left( W \right)  \right)
+\Theta\bullet(F(x)-W)\\
&\qquad\qquad\qquad\qquad\qquad  + \upsilon^\top x +  \nu^\top (\mathbf{e}-x) +\pi^\top (b-Ax) +\tau(s-\mathbf{e}^\top x),
\end{align*}
with $\mbox{ dom }\mathcal{L}= \mathbb{S}^{k,t}_{+} \times \mathbb{R}^{n}\times \mathbb{S}^k\times \mathbb{R}^{n}\times  \mathbb{R}^{n}\times  \mathbb{R}^{m}\times  \mathbb{R}$, where $ \mathbb{S}^{k,t}_{+}$ denotes the convex set of $k\times k$ positive semidefinite matrices with rank at least $t$.

The corresponding dual function is 
\begin{equation*}
\mathcal{L}^*(\Theta,\upsilon,\nu,\pi,\tau):=\sup_{W\in~  \mathbb{S}^{k,t}_{+},~x}\mathcal{L}(W,x,\Theta,\upsilon,\nu,\pi,\tau),
\end{equation*}
and the corresponding Lagrangian dual problem is
\[
 z_{{\mbox{\tiny DGFact}}}(C,s,t,A,b;F):=\inf\{\mathcal{L}^*(\Theta,\upsilon,\nu,\pi,\tau)~:~\upsilon\geq 0,~\nu\geq 0,~\pi\geq 0\}.
\]
We call $ z_{{\mbox{\tiny DGFact}}}:= z_{{\mbox{\tiny DGFact}}}(C,s,t,A,b;F)$
the \emph{generalized factorization bound}, as it generalizes the factorization bound for the special case of CMESP (see \cite{FLBook}).
We note that
\begin{align}
&\sup_{W\in \mathbb{S}^{k,t}_{+},~x} \textstyle
\left\{ \sum_{\ell=1}^t \log \left(\lambda_{\ell} \left( W \right)  \right)
+\Theta\bullet(F(x)-W) \right.\nonumber\\
&\left.\quad + ~ \upsilon^\top x +  \nu^\top (\mathbf{e}-x) +\pi^\top (b-Ax) +\tau(s-\mathbf{e}^\top x)\vphantom{\textstyle\sum_{\ell=1}^t }\right\} =\nonumber\\
& \sup_{W\in  \mathbb{S}^{k,t}_{+}} \textstyle 
\left\{ \sum_{\ell=1}^t \log \left(\lambda_{\ell} \left( W \right)  \right)
-\Theta\bullet W \right\}\label{supnonlin}\\
& \quad  +~ \sup_{x}\left\{ \Theta\bullet F(x) + \upsilon^\top x - \nu^\top x  - \pi^\top Ax - \tau \mathbf{e}^\top x +  \nu^\top \mathbf{e}  + \pi^\top b +\tau s\right\}.\label{suplin}
\end{align}
In Theorems \ref{dfactnonlin} and \ref{dfactlin}  we  analytically characterize the suprema in \eqref{supnonlin} and \eqref{suplin}. First, we state 
a useful lemma, used in the proof of Theorem \ref{dfactnonlin}.

\begin{lem}\label{lem:solinner}(see \cite[Lemma 11]{Weijun})
Given  $\hat \lambda_1\geq\hat \lambda_2\geq\cdots\geq \hat \lambda_k$ and $0< \hat \beta_1\leq\hat \beta_2\leq\cdots\leq \hat \beta_{k}$\,, we have that $\hat\theta_\ell:=\hat\beta_\ell$\,, for $\ell=1,\ldots,k$, solves 
\begin{equation*} 
 \min_{\substack{ \theta\in  \mathbb{R}^{k}_{+}\\\theta_1\leq\cdots\leq \theta_k}}\left\{\sum_{\ell=1}^k \hat\lambda_\ell\theta_\ell~:
 \sum_{\ell=j+1}^k \theta_\ell \leq  \sum_{\ell=j+1}^k \hat\beta_\ell\,,~j=1,\ldots,k-1, ~\sum_{\ell=1}^k \theta_\ell =  \sum_{\ell=1}^k \hat\beta_\ell\right\}.
\end{equation*}
\end{lem}

\begin{thm}\label{dfactnonlin} (similar to and generalizing \cite[Lemma 1]{Weijun}) For $\Theta\in\mathbb{S}^k$, we have
\begin{align}
&\sup_{W\in  \mathbb{S}^{k,t}_{+}}
\left(
\textstyle  \sum_{\ell=1}^t \log \left(\lambda_{\ell} \left( W \right)  \right)
-
W\bullet \Theta 
\right) 
\label{supdfact}\\
&\qquad=\left\{\begin{array}{ll} 
-~t -\!\!
\textstyle \sum_{\ell=k-t+1}^k \log\left(\lambda_{\ell} \left(\Theta\right)\right),
& \mbox{if}~\Theta \succ 0;\\
+\infty, &\mbox{otherwise.}
\end{array}\right.\nonumber
\end{align}
\end{thm}

\begin{proof} 
In what follows, we write $\hat\Theta$ for $\Theta$ to emphasize that it is fixed.
We first  assume that $\hat\Theta\succ 0$. 
 In this case,  it suffices to show that there is an optimal solution $\hat W$ to \eqref{supdfact}, such that 
\begin{equation}\label{identa}
\textstyle \sum_{\ell=1}^t \log \left(\lambda_{\ell} \left(\hat W \right)  \right)
-
\hat W\bullet \hat \Theta  ~=~
-t - 
\textstyle \sum_{\ell=k-t+1}^k \log (\lambda_{\ell} (\hat \Theta)).
\end{equation}
We  consider the spectral decomposition $
 \textstyle
\hat\Theta
=\sum_{\ell=1}^{k} \hat \beta_\ell \hat v_\ell \hat v_\ell^\top\,$,
where we assume that $0<\hat \beta_1\leq\hat \beta_2\leq\cdots\leq \hat \beta_{k}$\,.
We  define 
\begin{equation*} 
\textstyle \hat{W}:=\sum_{\ell=1}^{k} {\hat \lambda}_\ell \hat{v}_\ell \hat{v}_\ell^\top\,,
\end{equation*}
where 
\begin{equation*} 
\hat{\lambda}_\ell:=\left\{
\begin{array}{ll}
       1/\hat{\beta}_\ell\,,&\mbox{ for }1\leq \ell\leq t;\\[5pt]
       0, &\mbox{ for }t+1\leq \ell\leq k.
\end{array}\right.
\end{equation*}
Note  that $\hat \lambda_1\geq\hat \lambda_2\geq\cdots\geq \hat \lambda_k$
and that \eqref{identa} holds for $\hat W$. 

It remains to show that $\hat{W}$ is an optimal solution to \eqref{supdfact}. Let us consider 
the spectral 
decomposition of our matrix variable $W=\sum_{\ell=1}^{k} \lambda_\ell(W) u_\ell(W) u_\ell(W)^\top,
$ 
where we  assume that
$\lambda_1(W)\geq\lambda_2(W)\geq\cdots\geq \lambda_k(W)$. 

In the following we  consider $U(W):=(u_1(W),\ldots,u_k(W))$ and $\lambda(W):=(\lambda_1(W),\ldots,\lambda_k(W))^\top$. So, we have
\begin{align*}
W\bullet \hat\Theta &= (U(W) \Diag(\lambda(W)) U(W)^\top )\bullet \hat\Theta = \Diag(\lambda(W))\bullet (U(W)^\top \hat\Theta U(W)) \\
&= \Diag(\lambda(W))\bullet\Diag(\theta) = \textstyle\sum_{\ell=1}^k \lambda_\ell(W)\theta_\ell\,,
\end{align*}
where $\theta:=\diag(U(W)^\top \hat\Theta U(W))$. Then, defining the vector variable $\lambda\in\mathbb{R}^k$, and the  matrix variable $U\in\mathbb{R}^{k\times k}$, we replace $W$ with $U\Diag(\lambda)U^\top$ in \eqref{supdfact}, reformulating it as
\begin{equation}\label{sup2}
\sup_{\substack{ \lambda\in  \mathbb{R}^{k}_{+}\\\lambda_1\geq\cdots\geq \lambda_k}}\!\!\!
\left\{
 \sum_{\ell=1}^t \log \left(\lambda_{\ell}   \right)
-\!\! \min_{U,\theta\in\mathbb{R}^k_+}\left\{\sum_{\ell=1}^k \lambda_\ell\theta_\ell:\theta=\diag(U^\top \hat\Theta U), U^\top U=I\right\}\!
\right\}.
\end{equation}
Note that we are justified in writing minimum  rather than infimum 
in the inner problem in \eqref{sup2}, because the infimum is over a compact set. 

Finally, considering that  $U$ can be selected as any permutation matrix and $\lambda_1\geq\cdots\geq \lambda_k$ in the inner problem in \eqref{sup2},  we conclude from the Hardy-Littlewood-P\'olya rearrangement inequalities that $\theta_1\leq\cdots\leq \theta_k$ holds at its optimal solution.

Also considering that the eigenvalues of $U^\top \hat\Theta U$ and $\hat\Theta$ are the same, we see by 
the Schur's Theorem: majorization by eigenvalues,
that the following problem is a relaxation of the inner problem in \eqref{sup2}.

\begin{equation*} 
 \min_{\substack{ \theta\in  \mathbb{R}^{k}_{+}\\\theta_1\leq\cdots\leq \theta_k}}\!\!\left\{\sum_{\ell=1}^k \lambda_\ell\theta_\ell:
 \sum_{\ell=j+1}^k \theta_\ell \!\leq  \!\sum_{\ell=j+1}^k \hat\beta_\ell\,,
 1\leq j< k,
 ~\sum_{\ell=1}^k \theta_\ell =  \sum_{\ell=1}^k \hat\beta_\ell\right\}.
\end{equation*}
Now, using Lemma \ref{lem:solinner}, we conclude that $\sum_{\ell=1}^k \lambda_\ell\hat\beta_\ell$ is  a lower bound for the optimal value of the inner minimization problem in \eqref{sup2}. In fact, this lower bound is achieved by selecting  $\theta_\ell=\hat\beta_\ell$\,, for $\ell=1,\ldots,k$, and $U=(\hat v_1,\ldots,\hat v_k)$, which is then, an optimal solution for the inner problem.  Therefore, we may finally rewrite \eqref{sup2} as  
\begin{equation}\label{sup3}
\sup_
{\substack
{ 
\lambda\in  \mathbb{R}^{k}_{+}
\\\lambda_1\geq\cdots\geq \lambda_k}
}
\left\{\textstyle
 \sum_{\ell=1}^t \log \left(\lambda_{\ell}   \right)
- \sum_{\ell=1}^k \lambda_\ell\hat\beta_\ell
\right\}.
\end{equation}
The Lagrangian function associated to \eqref{sup3} is 
\[
\textstyle
\mathcal{L}(\lambda, \mu):=  \sum_{\ell=1}^t \log \left(\lambda_{\ell}   \right)
- \sum_{\ell=1}^k \lambda_\ell\hat\beta_\ell + \sum_{\ell=1}^{k}\mu_\ell\lambda_{\ell}\,,
\]
with $\mbox{ dom }\mathcal{L} = \mathbb{R}_+^{k,t}\times \mathbb{R}^{k}$, where $\mathbb{R}_+^{k,t}$ denotes the set of  vectors in $\mathbb{R}_+^k$ with at least $t$ positive components. We note that the Lagrangian function is concave in $\lambda$, so the supremum in \eqref{sup3} occurs where its gradient  with respect to $\lambda$ is equal to zero. Then, the optimality conditions for \eqref{sup3} are given by
\begin{align*}
&\frac{1}{\lambda_{\ell}} - \hat\beta_\ell +\mu_\ell =0,&&\mbox{for }\ell=1,\ldots,t,\\
&-\hat\beta_\ell +\mu_\ell=0,&&\mbox{for }\ell=t+1,\ldots,k,\\
&\mu_\ell\lambda_\ell=0,&&\mbox{for }\ell=1,\dots,k,\\
&\mu_\ell\geq 0,&&\mbox{for }\ell=1,\ldots,k,\\
&\lambda_1\geq\cdots\geq \lambda_k\geq 0.
\end{align*}
We see that the optimality conditions hold for $\lambda_\ell=\hat\lambda_\ell$\,, for $\ell=1,\ldots,k$, $\mu_\ell=0$, for $\ell=1,\ldots,t$ and $\mu_\ell=\hat\beta_\ell$\,, for $\ell=t+1,\ldots,k$. Therefore, $\hat\lambda_\ell$\,, for $\ell=1,\ldots,k$, is an optimal solution to \eqref{sup3}, and equivalently, to \eqref{sup2}. So, we finally see that $\hat{W}$ is an optimal solution to \eqref{supdfact}.

Next, we assume instead that $\hat\Theta\nsucc 0$. Then, we  consider the spectral 
decomposition $\hat \Theta:=\sum_{\ell=1}^k \hat\lambda_\ell \hat u_\ell \hat u_\ell^\top$
with $\hat \lambda_1\geq\hat \lambda_2\geq\cdots\geq \hat\lambda_p>0\geq \hat \lambda_{p+1}\geq\hat \lambda_{p+2}\geq\cdots\geq \hat\lambda_k$\,, where we have $p<k$. For a given $\alpha\in \mathbb{R}$, we define 
\[
\textstyle
\hat W:= \sum_{\ell=1}^{t-1}(\hat u_\ell \hat u_\ell^\top ) + \alpha\hat u_k \hat u_k^\top\,.
\]
We can verify that, for any $\alpha>0$,  $\hat W$ is a feasible solution for  \eqref{supnonlin} with objective value  
\begin{equation}\label{objw}
\textstyle
\log(\alpha) - \sum_{\ell=1}^{t-1} \hat\lambda_\ell + \alpha|\hat\lambda_k|.
\end{equation}
Clearly, \eqref{objw} goes to $+\infty$, as $\alpha\rightarrow +\infty$, which completes the proof.
\qed 
\end{proof}

\begin{thm}\label{dfactlin} For $(\Theta,\upsilon,\nu,\pi,\tau)\in \mathbb{S}^k\times \mathbb{R}^{n}\times  \mathbb{R}^{n}\times  \mathbb{R}^{m}\times  \mathbb{R}$, we have
\begin{align*}
&\sup_{x}\left(\Theta\bullet F(x) + \upsilon^\top x - \nu^\top x  - \pi^\top Ax - \tau \mathbf{e}^\top x +  \nu^\top \mathbf{e}  + \pi^\top b +\tau s \right)\\
&\quad = \left\{\begin{array}{ll}
\nu^\top \mathbf{e}  + \pi^\top b +\tau s,& \quad \mbox{if} ~ 
\diag(F \Theta F^\top) + \upsilon - \nu  - A^\top \pi - \tau\mathbf{e}=0;\\
+\infty, &\quad \mbox{otherwise.}
\end{array}\right.
\end{align*}
\end{thm}

\begin{proof}
    The result  follows from the fact that a linear function is bounded above 
only when it is identically zero.
\qed\end{proof}

Considering Theorems \ref{dfactnonlin} and \ref{dfactlin}, we see that the Lagrangian dual  of \ref{GFact} is equivalent to 
\[
\begin{array}{ll}
\multicolumn{2}{l}{z_{{\mbox{\tiny DGFact}}}(C,s,t,A,b;F)=} \\	
\qquad \min& -
 \sum_{\ell=k-t+1}^k 
\log\left(\lambda_{\ell} \left(\Theta\right)\right)
+ \nu^\top \mathbf{e}  + \pi^\top b +\tau s - t\\[4pt]
\qquad \mbox{s.t.}
     & \diag(F \Theta F^\top) + \upsilon - \nu  - A^\top \pi - \tau\mathbf{e}=0,\\[4pt]
& \Theta\succ 0, ~\upsilon\geq 0, ~\nu\geq 0, ~\pi\geq 0.
\end{array}
\tag{DGFact}\label{DGFact}
\]

From Lagrangian duality, we conclude that \ref{DGFact} is a convex-optimization problem. Nevertheless,  we note that \ref{GFact} is not generally a convex-optimization problem, so
we will not generally have strong duality between \ref{GFact} and \ref{DGFact}.

Next, we identify two properties for the generalized factorization bound that were similarly established for the factorization bound for  MESP in \cite{ChenFampaLee_Fact}. Specifically, we show that  the
generalized factorization bound for \ref{CGMESP} is invariant under multiplication of $C$ by a scale factor $\gamma$, up to the 
additive constant $-t\log \gamma$, and is also independent of the factorization of $C$. The proofs of the results are similar to the ones presented in \cite[Thm. 2.1 and Cor. 2.3]{ChenFampaLee_Fact} for MESP.

\begin{thm}\label{thm:scale}
For all $\gamma>0$ and factorizations $C=FF^\top$,
we have
\[
z_{{\mbox{\tiny\rm DGFact}}}(C,s,t,A,b;F) =  z_{{\mbox{\tiny\rm DGFact}}}(\gamma C,s,t,A,b; \sqrt{\gamma}F)  -t \log \gamma .
\]
\end{thm}

\begin{thm}\label{cor:FactDoesntMatter}
$z_{{\mbox{\tiny\rm DGFact}}}(C,s,t,A,b;F)$ does not depend on the chosen $F$.
\end{thm}

\cite{LeeLind2020} presents a
\emph{spectral bound}  for \ref{GMESP}, 
$\sum_{\ell=1}^t \log \lambda_\ell(C)$.
Next, we present a relation between the generalized factorization bound  and the spectral bound   for \ref{GMESP}. 
 This result generalizes the corresponding result for MESP (see \cite[Thm. 3.4.18]{FLBook}).

\begin{thm}\label{thm:almost_dominates}
Let $C\in\mathbb{S}^n_{+}$\,, with $r:=\rank(C)$, $0<t\leq r$, and $t\leq s < n$. Then,
for all factorizations $C=F F^\top$, we have
\[
 \textstyle
z_{{\mbox{\tiny\rm DGFact}}}(C,s,t,\cdot,\cdot\,\, ;F)
~-~ \sum_{\ell=1}^t  \log \lambda_\ell(C)~ \leq
~t\log\left(\frac{s}{t}\right).
\]
\end{thm}

\begin{proof}
Let $C=\sum_{\ell=1}^r \lambda_\ell(C) u_\ell u_\ell^\top$ be a spectral
decomposition of $C$.
By Theorem \ref{cor:FactDoesntMatter},
it suffices to take $F$ to be the 
symmetric matrix $\sum_{\ell=1}^r \sqrt{\lambda_\ell(C) } u_\ell u_\ell^\top$.
We consider the solution for \ref{DGFact} given by:   $\hat{\Theta}:=\frac{t}{s}\left(C^{\dagger} + \frac{1}{\lambda_r(C)}\left(I-CC^{\dagger}\right)\right)$, where $C^{\dagger}:=\sum_{\ell=1}^r \frac{1}{\lambda_\ell(C)}  u_\ell u_\ell^\top$ is the  Moore-Penrose pseudoinverse  of $C$,  $\hat{\upsilon}:=\frac{t}{s}\mathbf{e} - \diag(F\hat\Theta F^\top)$, 
$\hat{\nu}:=0$,  $\hat{\pi}:=0$, and $\hat{\tau}:=\frac{t}{s}$.
We can verify that the $r$ least eigenvalues of $\hat\Theta$ are $\frac{t}{s}\cdot\frac{1}{\lambda_1(C)},\frac{t}{s}\cdot\frac{1}{\lambda_2(C)},\ldots,\frac{t}{s}\cdot\frac{1}{\lambda_r(C)}$ and the $n-r$ greatest eigenvalues are all equal to  $\frac{t}{s}\cdot\frac{1}{\lambda_r(C)}$. 
Therefore, $\hat\Theta$ is positive definite. 

The equality constraint of \ref{DGFact} is clearly satisfied at this solution. Additionally, we can verify that $F\hat\Theta F^\top = \frac{t}{s}\sum_{\ell=1}^r u_\ell u_\ell^\top$\,. As  
$\sum_{\ell=1}^r u_\ell u_\ell^\top \preceq I$, we conclude that  $\diag(F\hat\Theta F^\top)\leq \frac{t}{s}\mathbf{e}$. Therefore, $\hat{\upsilon}\geq 0$, and the solution constructed is a feasible solution to \ref{DGFact}. Finally, we can see that the   objective value  
of this solution is equal to the spectral bound added to  $t\log(s/t)$. The result then follows.
\qed\end{proof}

\begin{rem}
Considering Theorem \ref{thm:almost_dominates},
for $t=s-\kappa$, with constant integer $\kappa\geq 0$,
$\lim_{s\rightarrow \infty} t\log(s/t) =\kappa$. 
Therefore, in this limiting regime, the generalized factorization bound 
is no more than an additive constant worse than the 
spectral bound. 
\end{rem}

Considering Theorem \ref{thm:almost_dominates}, 
we will see next that key quantities are discrete concave in $t$,
in such a way that we get a concave upper bound that only depends on $t$ and $s$
for the difference of two discrete concave upper bounds (which depend on  $C$ as well). 
  We note that these next two theorems have no counterparts for CMESP nor CD-Opt, where $t$ cannot vary
independently of~$s$.

\begin{thm}\label{thm:concavity}
\phantom{.}
\begin{enumerate}
\item[\rm(a)] $t\log\left(\frac{s}{t}\right)$ is (strictly) concave in $t$ on 
$\mathbb{R}_{++}$\,;    \item[\rm(b)] $ \sum_{\ell=1}^t  \log \lambda_\ell(C)~$ is discrete concave in $t$ on $\{1,2,\ldots,r\}$;
\item[\rm(c)] $z_{{\mbox{\tiny\rm DGFact}}}(C,s,t,A,b\,\, ;F)$  is discrete concave in $t$ on $\{1,2,\ldots,k\}$.
\end{enumerate}
\end{thm}

\begin{proof}
For (a), we see that the second derivative of the function is $-1/t$, which is negative on $\mathbb{R}_{++}$.
For (b), it is easy to check that discrete concavity is equivalent to $\lambda_t(C)\geq \lambda_{t+1}(C)$\,, for all integers $t$ satisfying 
$1\leq t<r$,
which we obviously have. For (c), first we observe that we can view 
$z_{{\mbox{\tiny DGFact}}}(C,s,t,A,b\,\, ;F)$ as the point-wise minimum 
of 
$- \sum_{\ell=k-t+1}^k 
\log\left(\lambda_{\ell} \left(\Theta\right)\right)
+ \nu^\top \mathbf{e}  + \pi^\top b +\tau s - t$
over the points in the convex feasible region of \ref{DGFact}.
So it suffices to demonstrate that the function 
$- \sum_{\ell=k-t+1}^k 
\log\left(\lambda_{\ell} \left(\Theta\right)\right)
+ \nu^\top \mathbf{e}  + \pi^\top b +\tau s - t$ is discrete concave in $t$.
It is easy to check that this is equivalent to $\lambda_{k-t+1}(C)\geq \lambda_{k-t+2}(C)$\,, for all integers $t$ satisfying $1<t \leq k$,
which we obviously have.
\qed
\end{proof}

Theorem \ref{thm:concavity}, part (c) is very interesting, in connection with the
motivating application to PCA. Using convexity, we can compute upper bounds 
on the value of changing the number $t$ of dominant principal components considered, 
without having to actually solve  further instances of \ref{DGFact} (of course, solving those would give better upper bounds). 

More generally than the spectral bound for \ref{GMESP}, \cite{LeeLind2020} defines the \emph{(Lagrangian) spectral bound} for \ref{CGMESP} as
\[
v^* := \min_{\pi\in \mathbb{R}^m_+} v(\pi) := 
\sum_{\ell=1}^t\log \lambda_\ell\left( D_\pi C D_\pi \right) + \pi^\top b ~- 
\min_{K\subset N,|K|=s-t}
\sum_{j\in K}\sum_{i\in M} \pi_i a_{ij}\,,
\]
where $D_\pi\in \mathbb{S}^n_{++}$ is  the diagonal matrix defined by
\[
D_\pi[\ell,j]:= 
\left\{
  \begin{array}{ll}
    \exp\{-\frac{1}{2}\sum_{i=1}^m \pi_i a_{ij}\},& \hbox{ for $\ell=j$;} \\
    0,&\hbox{ for $\ell\not=j$.}
  \end{array}
\right.
\]

Extending Theorem \ref{thm:concavity}, part (b), we have the following result. 
\begin{thm}\label{thm:disc}
\phantom{.}
   $v^*$  is discrete concave in $t$ on $\mathbb{Z}_{++}$\,.
\end{thm}

\begin{proof}
Because $v^*$ is a point-wise minimum, we need only 
demonstrate that 
$v(\pi)$ is discrete concave in $t$, for each fixed $\pi \in \mathbb{R}^m_+$\,.
Using the same reasoning as for 
the proof of  Theorem \ref{thm:almost_dominates}, part (b),
we can see that 
$\sum_{\ell=1}^t\log \lambda_\ell\left( D_\pi C D_\pi \right)$
is discrete concave in $t$. So, it is enough to
demonstrate that 
\begin{align*}
&    \textstyle
\min\left\{\sum_{j\in K}\sum_{i\in M} \pi_i a_{ij} ~:~ K\subset N,|K|=s-t\right\}\\
&\qquad = \textstyle\min
\left\{ 
\sum_{j\in N} \left(\sum_{i\in M} \pi_i a_{ij}\right) x_j ~:~ \mathbf{e}^\top x = s-t,\, 0\leq x \leq \mathbf{e},~ x\in\mathbb{Z}^n
\right\}.
\end{align*}
is discrete convex in $t$. 
Because of total unimodularity, we can drop the integrality requirement,
and we see that the last expression is equivalently a minimization linear-optimization problem, which is 
well-known to be convex in the right-hand side (see \cite[Chap. 6]{LeeLP}, for example), which is affine in $t$. The result follows. 
\qed
\end{proof}

In the next theorem, we present for \ref{CGMESP}, a key result to enhance the application of B\&B algorithms to discrete optimization problems with convex relaxations. The principle described in the theorem is called \emph{variable fixing}, and has been successfully applied to MESP (see \cite{AFLW_IPCO,AFLW_Using}). The similar proof of the theorem for MESP can be found in \cite[Theorem 3.3.9]{FLBook}.

\begin{thm} \label{thm:fixFact}
 Let
\begin{itemize}
    \item\!$\mbox{LB}$ be the objective-function value of a feasible solution for \ref{CGMESP},
    \item\!$(\!\hat{\Theta},\hat{\upsilon},\hat{\nu},\hat{\pi},\hat{\tau}\!)$ be a feasible solution for \ref{DGFact} with objective-function value~$\hat{\zeta}$.
\end{itemize}
Then, for every optimal solution $x^*$ 
for \ref{CGMESP}, we have:
\[
\begin{array}{ll}
x_j^*=0, ~ \forall ~ j\in N \mbox{ such that } \hat{\zeta}-\mbox{LB} < \hat{\upsilon}_j\thinspace,\\
x_j^*=1, ~ \forall ~ j\in N \mbox{ such that } \hat{\zeta}-\mbox{LB} < \hat{\nu}_j\thinspace.\\
\end{array}
\]
\end{thm}

 
\section{Duality for \ref{DGFact}}\label{sec:DDFact}

Although it is possible to directly solve \ref{DGFact}
to calculate the generalized factorization bound, it is computationally more attractive  to
work with the Lagrangian dual of 
\ref{DGFact}. In the following, we construct this dual formulation.

Consider the Lagrangian function corresponding to \ref{DGFact}, after eliminating the slack variable $\upsilon$,
\begin{align*}
&\mathcal{L}(\Theta,\nu,\pi,\tau,x,y,w):=-
\textstyle 
\sum_{\ell=k-t+1}^k \log\left(\lambda_{\ell} \left(\Theta\right)\right)
+ \nu^\top \mathbf{e}  + \pi^\top b +\tau s - t\\
     &\qquad\qquad\qquad\qquad\qquad +x^\top\left(\diag(F \Theta F^\top)  - \nu  - A^\top \pi - \tau\mathbf{e}\right) -y^\top\nu - w^\top\pi,
\end{align*}
with $\mbox{ dom }\mathcal{L}=  \mathbb{S}_{++}^k\times  \mathbb{R}^{n}\times  \mathbb{R}^{m}\times  \mathbb{R} \times \mathbb{R}^{n}\times \mathbb{R}^{n}\times  \mathbb{R}^{m}$.

The corresponding dual function is
\begin{equation*} 
\mathcal{L}^*(x,y,w):=\inf_{\Theta\in~\mathbb{S}^{k}_{++},\nu,\pi,\tau}\mathcal{L}(\Theta,\nu,\pi,\tau,x,y,w),
\end{equation*}
and the  Lagrangian dual problem of \ref{DGFact} is
\[
 z_{{\mbox{\tiny DDGFact}}}(C,s,t,A,b;F):=\max\{\mathcal{L}^*(x,y,z)~:~x\geq 0,~y\geq 0,~w\geq 0\}.
\]
We note that \ref{DGFact} has a strictly-feasible solution (e.g., given by $(\hat \Theta :=I,\,\hat \upsilon:=0,\,\hat \nu:=\diag( F F^\top)=\diag(C),\,\hat \pi:=0,\,\hat \tau:=0)$). Then, Slater's condition holds for \ref{DGFact} and   we are justified 
to use maximum in the formulation of the Lagrangian dual problem, rather than supremum, as the optimal value of the Lagrangian dual problem is attained.

We have  that 
\begin{align}
&\inf_{\Theta\in~\mathbb{S}^{k}_{++},\nu,\pi,\tau}\mathcal{L}(\Theta,\nu,\pi,\tau,x,y,w) =  \nonumber\\
&\quad\qquad \inf_{\Theta\in~\mathbb{S}^{k}_{++}} \left\{- 
\textstyle 
\sum_{\ell=k-t+1}^k \log\left(\lambda_{\ell} \left(\Theta\right)\right)
 +x^\top \diag(F \Theta F^\top)- t\right\}  \label{infnonlin} \\
&\qquad\qquad\quad + 
 \inf_{\nu,\pi,\tau}
 \left\{\nu^\top (\mathbf{e} - x - y)  + \pi^\top( b - Ax - w) +\tau (s - \mathbf{e}^\top x)\right\}.\label{inflin}
\end{align}

Next, we discuss the infima in \eqref{infnonlin} and \eqref{inflin}, for which
 the following lemma brings a fundamental result.

\begin{lem}\label{Ni13}(see \cite[Lemma 13]{Nikolov})
 Let $\lambda\in\mathbb{R}_+^k$ with $\lambda_1\geq \lambda_2\geq \cdots\geq \lambda_k$ and let $0<t\leq k$. There exists a unique integer $\iota$, with $0\leq \iota< t$, such that
 \begin{equation}\label{reslemiota}
 \lambda_{\iota}>\frac{1}{t-\iota}
 \textstyle 
 \sum_{\ell=\iota+1}^k \lambda_{\ell}\geq \lambda_{\iota +1}\,,
 \end{equation}
 with the convention $\lambda_0=+\infty$.
\end{lem}

Suppose that  $\lambda\in\mathbb{R}^k_+$\,, and assume that 
$\lambda_1\geq\lambda_2\geq\cdots\geq\lambda_k$\,. Given integer $t$ with $0<t\leq k$,
let $\iota$ be the unique integer defined by Lemma \ref{Ni13}. We define
\begin{equation}\label{def:phi}
\phi_t(\lambda):=
\textstyle
\sum_{\ell=1}^{\iota} \log\left(\lambda_\ell\right) + (t - \iota)\log
\left(\frac{1}{t-{\iota}} 
\textstyle
\sum_{\ell=\iota+1}^{k}
\lambda_\ell\right).
\end{equation}
Also, for $X\in\mathbb{S}_{+}^k$\,, we define 
$\Gamma_t(X):= \phi_t(\lambda(X))$.

Considering the definition of $\Gamma_t$, we analytically characterize the infimum in \eqref{infnonlin} in Theorem \ref{ddfactnonlin}. 

\begin{thm}\label{ddfactnonlin}(see \cite[Lemma 16]{Nikolov}) For $x\in\mathbb{R}^n$, we have
\begin{align}
&\inf_{\Theta\in~\mathbb{S}^{k}_{++}} -
\textstyle
\sum_{\ell=k-t+1}^k 
\log\left(\lambda_{\ell} \left(\Theta\right)\right)
 +x^\top \diag(F \Theta F^\top) - t\nonumber \\[2pt]
& \quad \quad =\left\{\begin{array}{ll} 
\Gamma_t(F(x)),
&\quad \mbox{if}~ F(x)\succeq 0 \mbox{ and }\rank(F(x))\geq t;\\
-\infty, &\quad \mbox{otherwise.}
\end{array}\right.\nonumber
\end{align}
\end{thm}

\begin{proof}  
In the proof, we write $\hat x$ for $x$ to emphasize that it is fixed.
We have
\begin{equation*}
\inf_{\Theta\in~\mathbb{S}^{k}_{++}} -
\textstyle
\sum_{\ell=k-t+1}^k 
\log\left(\lambda_{\ell} \left(\Theta\right)\right)
 +\hat{x}^\top \diag(F \Theta F^\top) - t
 \end{equation*}
 \begin{equation}
 \quad =\inf_{\Theta\in~\mathbb{S}^{k}_{++}} -
 \textstyle
 \sum_{\ell=k-t+1}^k
 \log\left(\lambda_{\ell} \left(\Theta\right)\right)
 +F(\hat{x})\bullet \Theta - t, \label{convp}
\end{equation}
where $F(\hat{x})=\sum_{i=1}^n F_{i\cdot}^\top F_{i\cdot} ~\hat{x}_i$ and $\lambda_\ell(\Theta)$ denotes the $\ell$-th greatest eigenvalue of $\Theta$. 

We first assume that $F(\hat{x})\succeq 0$ and $\hat r:=\rank(F(\hat{x}))\geq t$. In this case,  it suffices to show that there is an optimal solution $\hat \Theta$ to \eqref{convp}, such that 
\begin{equation}\label{Gammaident}
\textstyle
-\sum_{\ell=k-t+1}^k 
\log (\lambda_{\ell} (\hat\Theta))
 +F(\hat{x})\bullet \hat\Theta - t=\Gamma_t(F(\hat{x})).
\end{equation}

We  consider the spectral 
decomposition of $F(\hat{x})$,
$
F(\hat{x})=
\textstyle
\sum_{\ell=1}^{k} \hat \lambda_\ell \hat u_\ell \hat u_\ell^\top\,,
$
with $\hat \lambda_1\geq\hat \lambda_2\geq\cdots\geq \hat \lambda_{\hat r}>\hat \lambda_{\hat{r}+1}=\cdots=\hat \lambda_k=0$,
and the spectral 
decomposition of our matrix variable $\Theta$ as 
$
\textstyle
\Theta=\sum_{\ell=1}^{k} \beta_\ell(\Theta) v_\ell(\Theta) v_\ell(\Theta)^\top,
$
where here, we conveniently assume that
$\beta_1(\Theta)\leq\beta_2(\Theta)\leq\cdots\leq \beta_k(\Theta)$. 
Now, we  define 
$
\textstyle
\hat{\Theta}:=\sum_{\ell=1}^{k} {\hat \beta}_\ell \hat{u}_\ell \hat{u}_\ell^\top\,,
$
where 
\begin{equation*} 
\hat{\beta}_\ell:=\left\{
\begin{array}{ll}
       1/\hat{\lambda}_\ell\,,&\mbox{ for }~1\leq \ell\leq \hat{\iota};\\
       1/\hat{\delta},&\mbox{ for }~\hat{\iota}<\ell\leq \hat{r};\\
       (1+\epsilon)/\hat{\delta},&\mbox{ for }~\hat{r}<\ell\leq k,
\end{array}\right.
\end{equation*}
 where $\epsilon$ is any positive number,  $\hat{\iota}$ is the unique integer defined  in Lemma \ref{Ni13} for $\lambda_\ell=\hat{\lambda}_\ell$\,, and
\begin{equation}\label{defhatdelta}
\hat \delta:=\textstyle \frac{1}{t-\hat \iota}
\textstyle
\sum_{\ell=\hat \iota+1}^{k}\hat \lambda_\ell\,.
\end{equation}
Note  that $0<\hat \beta_1\leq\hat \beta_2\leq\cdots\leq \hat \beta_k$\,.

From Lemma \ref{Ni13}, we have that $\hat\iota<t$. Then, as $t\leq \hat{r}$,  
\begin{equation}\label{sumbeta}
\begin{array}{ll}
-\textstyle \sum_{\ell=1}^t \log\left(\hat{\beta}_{\ell}\right) &= -\textstyle \sum_{\ell=1}^{\hat{\iota}} \log\left(\frac{1}{\hat{\lambda}_{\ell}}\right) -\textstyle \sum_{\ell=\hat{\iota}+1}^t \log\left(\frac{1}{\hat{\delta}}\right) \\
&= \textstyle \sum_{\ell=1}^{\hat{\iota}} \log\left(\hat{\lambda}_{\ell}\right) + (t-\hat{\iota})\log(\hat{\delta})= \Gamma_t(F(\hat{x})),
\end{array}
\end{equation}
where the last identity follows from the definition of $\Gamma_t$ and \eqref{defhatdelta}.

Moreover, considering that $\hat\lambda_\ell=0$, for $\ell>\hat{r}$, from \eqref{defhatdelta}, we also have
\begin{equation}\label{fdottheta}
F(\hat{x})\bullet \hat{\Theta}=\sum_{\ell=1}^{\hat{r}} \hat{\lambda}_{\ell}\hat{\beta}_{\ell} = \sum_{\ell=1}^{\hat{\iota}} \hat{\lambda}_{\ell}\frac{1}{\hat{\lambda}_{\ell}}  + \sum_{\ell={\hat{\iota}}+1}^{\hat{r}} \hat{\lambda}_{\ell}\frac{1}{\hat \delta} =\hat\iota + \frac{1}{\hat \delta}\sum_{\ell=\hat{\iota}+1}^{\hat{r}} \hat{\lambda}_{\ell} = \hat{\iota} + (t-\hat{\iota}) = t,
\end{equation}
which, together with \eqref{sumbeta},  proves of the identity in \eqref{Gammaident}. 

Therefore,  it remains to show that $\hat{\Theta}$ is an optimal solution to \eqref{convp}, which we do in the following by showing that the subdifferential of the objective function of \eqref{convp} contains 0 at $\hat{\Theta}$. 
The second term in the objective function of \eqref{convp} is differentiable with respect to $\Theta$. In fact,
$
\frac{\mathrm{d}}{\mathrm{d} \Theta}  \left(F(\hat{x})\bullet \Theta \right)= F(\hat{x}).
$
So, we  need to show that 
$
-F(\hat{x}) \in \partial \Delta_t( \hat \Theta), 
$
where
$
\Delta_t(\Theta):=-\textstyle \sum_{\ell=k-t+1}^k \log\left(\lambda_{\ell} \left(\Theta\right)\right).
$

We recall that  $\Theta=\sum_{\ell=1}^{k} \beta_\ell(\Theta) v_\ell(\Theta) v_\ell(\Theta)^\top$, with  $\beta_1(\Theta)\leq\beta_2(\Theta)\leq\cdots\leq \beta_k(\Theta)$.
Then, if $t\in\mathcal{E}:=\{\hat\iota+1,\ldots, \hat r \}$, where $\mathcal{E}$ is defined so that  $\beta_{\hat\iota}(\Theta)<\beta_{\hat\iota+1}(\Theta)=\cdots=\beta_{\hat r }(\Theta)<\beta_{\hat r+1}(\Theta)$, we have 
that the subdifferential of $\Delta_t$ at $\Theta$ is given by
\[
\partial \Delta_t(\Theta) =
\left\{
V\Diag(\lambda)V^\top : 
\lambda \in \mathcal{U}(\beta(\Theta)),~
V^\top V = I,~ \Theta=V \Diag(\beta(\Theta))V^\top\right\},
\]
where  $\mathcal{U}(\beta(\Theta))$
is the convex hull of the subgradients 
\[
 \left(-\frac{1}{\beta_1(\Theta)},-\frac{1}{\beta_2(\Theta)},\ldots,-\frac{1}{\beta_{\hat\iota}(\Theta)}, g_{\hat\iota+1},\ldots,g_{\hat r},0,\ldots,0\right)^\top,
\]
with
\[
g_i :=
\left\{
  \begin{array}{ll}
  -1/\beta_i(\Theta), & \hbox{ for $i\in \mathcal{F}$;} \\
   0 , & \hbox{  for $i\in \mathcal{E}\setminus \mathcal{F}$,}
  \end{array}
\right.
\]
for all sets  $\mathcal{F}$, such that  
$\mathcal{F}\subset \mathcal{E}$ and $|\mathcal{F}|=t-\hat\iota$.

Thus, we may write
\[
\begin{array}{ll}
\mathcal{U}&(\beta(\Theta))=\Big\{\lambda:\lambda_i=-1/\beta_i(\Theta),\quad1\leq i\leq \hat\iota,\\[5pt]
&(\lambda_{\hat\iota +1}, \ldots,\lambda_{\hat r})\in\textstyle\frac{-1}{\beta_{\hat r}(\Theta)}\conv\left\{\omega\in\mathbb{R}^{\hat r - \hat\iota}:\displaystyle \mathbf{e}^\top \omega=t-\hat\iota,~\omega\in\{0,1\}^{\hat r - \hat\iota}\right\},\\[5pt]
&\lambda_{i}=0,\quad \hat r+1\leq i\leq k\Big\}.
\end{array}
\]
Finally, considering that $-F(\hat{x})=-\sum_{\ell=1}^{k} \hat\lambda_\ell \hat u_\ell \hat u_\ell^\top$\,, $\hat \Theta=\sum_{\ell=1}^{k} \hat\beta_\ell \hat u_\ell \hat u_\ell^\top$\,, $\hat \delta=\frac{1}{t-\hat \iota}\sum_{\ell=\hat \iota+1}^{\hat r}\hat \lambda_\ell$\,, 
$
\hat\lambda=\left({\hat\lambda_1},\ldots,{\hat\lambda_{\hat\iota}}, \hat\lambda_{\hat\iota+1},\ldots,\hat\lambda_{\hat r},0,\ldots,0\right)^\top
$, and
\[
\hat\beta=\Big(\frac{1}{\hat\lambda_1},\ldots,\frac{1}{\hat\lambda_{\hat\iota}},\underbrace{ \frac{1}{\hat\delta},\ldots,\frac{1}{\hat\delta}}_{\mathcal{E}},\frac{1+\epsilon}{\hat\delta},\ldots,\frac{1+\epsilon}{\hat \delta}\Big)^\top,
\]
for a given $\epsilon>0$, and noting that 
\begin{align*}
&\hat\delta\times \conv\left\{\omega\in\mathbb{R}^{\hat r - \hat\iota}:\displaystyle \mathbf{e}^\top \omega=t-\hat\iota,~\omega\in\{0,1\}^{\hat r - \hat\iota}\right\}\\
&\qquad = \left\{\omega\in\mathbb{R}^{\hat r - \hat\iota}: \mathbf{e}^\top \omega=(t-\hat\iota)\hat\delta,~0\leq \omega_i\leq \hat\delta, ~i=1,\ldots,\hat r - \hat\iota\right\},
\end{align*}
we can see that $-\hat\lambda \in \mathcal{U}(\hat\beta)$, and, therefore, $-F(\hat{x}) \in \partial \Delta_t( \hat \Theta)$, which completes the proof for our first case.

Now, we consider the second case, where we have
$F(\hat{x})\nsucceq 0$ or $\rank(F(\hat{x}))< t$. We use the   spectral 
decomposition 
$F(\hat{x})=\sum_{\ell=1}^k\hat\lambda_\ell \hat u_\ell \hat u_\ell^\top$\,, and define $\hat\lambda_0:=+\infty$. We assume that 
\[
\hat \lambda_0\geq\hat \lambda_1\geq\hat\lambda_2\geq \cdots\geq \hat \lambda_p>0=\hat \lambda_{p+1}=\cdots=\hat \lambda_q>\hat \lambda_{q+1}\geq\hat \lambda_{q+2}\geq \cdots\geq \hat\lambda_k\,,
\]
for some $p,q\in[0,k]$.

In this case, it is straightforward to  see that $\hat\Theta:=\sum_{\ell=1}^{p} (1/\hat \lambda_\ell) \hat{u}_\ell \hat{u}_\ell^\top + \sum_{\ell=p+1}^{k} \alpha~\hat{u}_{\ell} \hat{u}_{\ell}$ is a feasible solution to \eqref{convp} for any $\alpha> 0$. When $\alpha>1/\hat\lambda_p$\,, the objective value of \eqref{convp} at $\hat\Theta$ becomes 
\begin{equation}\label{objconvp}
-\textstyle\sum_{\ell=1}^p \log \frac{1}{\hat\lambda_\ell} -\sum_{\ell=p+1}^t \log \alpha + p - \sum_{\ell=q}^k \alpha|\hat\lambda_\ell|.
\end{equation}
From the assumptions on $F(\hat{x})$, we see that either $q<k$ (so $\lambda_k<0$) or $p<t$. Therefore, we can verify that, in both cases,  as $\alpha\rightarrow +\infty$, \eqref{objconvp} goes to $-\infty$, which completes our proof.
\qed\end{proof}

Finally,   we  analytically characterize the infimum in  \eqref{inflin} in Theorem \ref{ddfactlin}. Its  proof follows from the fact that a linear function is bounded below only when it is identically zero.

\begin{thm}\label{ddfactlin} For $(x,y,w)\in  \mathbb{R}^{n}\times \mathbb{R}^{n}\times  \mathbb{R}^{m}$, we have
\begin{align*}
&\inf_{\nu,\pi,\tau}
 \nu^\top (\mathbf{e} - x - y)  + \pi^\top( b - Ax - w) +\tau (s - \mathbf{e}^\top x) \\
&\quad \quad = \left\{\begin{array}{cl}
0,& \quad \mbox{if  } ~ \mathbf{e} - x - y=0,~b - Ax - w=0,~s - \mathbf{e}^\top x=0;\\
-\infty, &\quad \mbox{otherwise.}
\end{array}\right.
\end{align*}
\end{thm}

Considering Theorems \ref{ddfactnonlin} and \ref{ddfactlin}, the Lagrangian dual of \ref{DGFact} is equivalent to 

\begin{equation}\tag{DDGFact}\label{DDGFact}
\begin{array}{ll}
&z_{{\mbox{\tiny DDGFact}}}(C,s,t,A,b;F)=\max \left\{\Gamma_t(F(x)) ~:~ \mathbf{e}^\top x =s,~Ax\leq b,~ 0\leq x\leq \mathbf{e}\right\},
\end{array}
\end{equation}

From Lagrangian duality, we have that \ref{DDGFact} is a convex-optimization problem. Moreover, because \ref  {DGFact} has a strictly-feasible solution,  if  \ref{DDGFact} has a feasible solution with finite objective value, 
then we have strong duality between \ref{DDGFact} and \ref{DGFact}. 

Finally, for developing a  nonlinear-optimization algorithm
for \ref{DDGFact}, we consider in the next theorem,  an expression for the gradient of its
 objective function. The proof is similar to the one presented for MESP in 
\cite[Thm. 2.10]{ChenFampaLee_Fact}.

\begin{thm}\label{thm:supgrad}
Let $F(\hat x)=\sum_{\ell=1}^{k} \hat\lambda_{\ell} \hat u_{\ell} \hat u_{\ell}^{\top}$ be a spectral decomposition of $F(\hat x)$.
Let $\hat \iota$ be the value of
$\iota$ in Lemma \ref{Ni13},
where $\lambda$ in Lemma \ref{Ni13}
is $\hat{\lambda}:=\lambda(F(\hat x))$.
 If $\frac{1}{t-\hat \iota}\sum_{\ell=\hat \iota+1}^k \hat \lambda_{\ell}> \hat \lambda_{\hat\iota +1}$\,, then, for 
$j=1,2,\ldots,n$,
\begin{equation*}
      \frac{\partial}{\partial x_j} \Gamma_t(F(\hat x))
    =  \sum_{\ell=1}^{\hat\iota} \frac{1}{\hat\lambda_\ell}(F_{j\cdot} \hat u_\ell)^2
    + \sum_{\ell=\hat\iota+1}^{k} \frac{t-\hat\iota}{\sum_{i=\hat\iota+1}^{k}\hat\lambda_i}  (F_{j\cdot} \hat u_\ell)^2\thinspace.
\end{equation*}
\end{thm}
As observed in \cite{ChenFampaLee_Fact}, without the technical condition $\frac{1}{t-\hat \iota}\sum_{\ell=\hat \iota+1}^k \hat \lambda_{\ell}> \hat \lambda_{\hat\iota +1}$\,,
the formula above still gives a subgradient of $\Gamma_t$
(see \cite{Weijun}).

 As we already pointed out, \ref{GFact} is an exact but nonconvex  relaxation of \ref{CGMESP}.
 On the other hand,  \ref{DDGFact} is a convex relaxation for \ref{CGMESP}, 
which is non-exact generally, except for the important case of  $t=s$ when it becomes exact.
In Theorem \ref{thm:propphi}, we present   properties of the function $\phi_t$ defined in \eqref{def:phi}, which show that,  the relaxation is generally non-exact for $t<s$ and it is always exact for $t=s$. In Lemma \ref{lem:iota_t} we prove the relevant facts for their understanding.

\begin{lem}\label{lem:iota_t}
 Let $\lambda\in\mathbb{R}_+^n$ with $\lambda_1\geq \lambda_2\geq \cdots\geq
 \lambda_\delta> \lambda_{\delta+1}=\cdots=
 \lambda_r>\lambda_{r+1}=\cdots
 =\lambda_n=0$. Then,

 \begin{enumerate}
     \item[\rm(a)] For $t=r$, the $\iota$ satisfying
 \eqref{reslemiota} is
 precisely $\delta$. 
 
 \item[\rm(b)] For $r<t\leq n$, the $\iota$ satisfying
 \eqref{reslemiota} is
 precisely $r$. 
 \end{enumerate}
\end{lem}

\begin{proof}
For (a), the result follows  because
\[
\textstyle\frac{1}{r-\delta}\sum_{\ell=\delta+1}^n \lambda_{\ell} = \lambda_{\delta+1} \quad\mbox{and}\qquad \lambda_\delta> \lambda_{\delta+1}\,.
\]
For (b), the result follows because 
\[
\textstyle\frac{1}{t-r}\sum_{\ell=r+1}^n \lambda_{\ell} = 0 =\lambda_{r+1} \quad\mbox{and}\qquad \lambda_r> \lambda_{r+1}\,.
\]
\qed
\end{proof}

\begin{thm}\label{thm:propphi}
Let $\lambda\in\mathbb{R}_+^n$ with $\lambda_1\geq \lambda_2\geq \cdots\geq
 \lambda_r>\lambda_{r+1}=\cdots
 =\lambda_n=0$. Then,
 \[
 \begin{array}{lll}
 \mbox{\rm(a)}&\textstyle\phi_{t}(\lambda) >  \sum_{\ell=1}^{t} \log\left(\lambda_\ell\right), & \mbox{  for $0<t<r$,}\\[2pt]
\mbox{\rm(b)}&\phi_t(\lambda)=\sum_{\ell=1}^{t} \log\left(\lambda_\ell\right),& \mbox{  for $t= r$,}\\[2pt]
\mbox{\rm(c)}&\phi_t(\lambda)=-\infty, & \mbox{  for $r< t\leq n$.}
 \end{array}
 \]
 where we use $\log(0)=-\infty$. 
\end{thm}

\begin{proof}
Part (a) follows from:
\[
(t-\iota)\log\left(\!{\textstyle\frac{1}{t-\iota}}\sum_{\ell=\iota+1}^{n}\lambda_{\ell}\!\right) > (t-\iota)\log\left(\!{\textstyle\frac{1}{t-\iota}}\sum_{\ell=\iota+1}^{t}\lambda_{\ell}\!\right) \geq 
(t-\iota)\left(\!\frac{\sum_{\ell=\iota+1}^{t}\log (\lambda_{\ell})}{t-\iota}\!\right).
\]
Parts (b) and (c) follow from Lemma \ref{lem:iota_t}.
\qed\end{proof}

\begin{rem}
     We can conclude that  $z(C,s,t,A,b;F)\leq z_{{\mbox{\tiny DDGFact}}}(C,s,t,A,b;F)$ from the use of Lagrangian duality. But we note that Theorem \ref{thm:propphi}  gives an alternative and direct proof for this result, besides showing in part (a), that the inequality is strict whenever the rank of any optimal submatrix 
     $C[S,S] $ for \ref{CGMESP} is greater than  $t$. 
     We note that for CMESP, where $t=s$, part (a) cannot happen, and this
     is why the relaxation is exact for CMESP. 
\end{rem}

Finally, we note that to apply the variable-fixing procedure described in Theorem \ref{thm:fixFact} in a  B\&B algorithm to solve \ref{CGMESP}, we need a feasible solution for \ref{DGFact}. To avoid incorrect variable fixing based on dual information from \emph{near}-optimal  solutions to \ref{DDGFact}, following  \cite{Weijun}, we will show  how to rigorously construct a  feasible solution for \ref{DGFact} from  a feasible solution $\hat x$ of \ref{DDGFact}  with finite objective value, with the goal of producing a small gap.

Although in \ref{CGMESP},  the lower bounds are all zero  and the upper bounds are all one, we will consider the more general problem with lower and upper bounds on the variables given respectively by $l,c\in\{0,1\}^n$,
with $l\leq c$. 
So, we consider the constraints $l\leq\! x\!\leq c$ in \ref{DDGFact}, instead of $0\leq\! x\!\leq \mathbf{e}$. The motivation for this, is to derive the technique to fix variables at any subproblem considered during the execution of the B\&B algorithm, when some of the variables may already be fixed (i.e., $l_i=c_i$ fixes $x_i$). Instead of redefining the problem with fewer variables in our numerical experiments, we found that it was more efficient to change the upper bound $c_i$ from one to zero, when variable $i$ is fixed at zero in a subproblem, and similarly,  change the lower bound $l_i$ from zero to one, when variable $i$ is fixed at one. 

Then, to construct the dual solution, we  consider a feasible solution $\hat x$ of \ref{DDGFact}, with the constraints $0\leq x\leq \mathbf{e}$  replaced by $l\leq x\leq c$, and the spectral
decomposition $F(\hat{x})=\sum_{\ell=1}^{k} \hat \lambda_\ell \hat u_\ell \hat u_\ell^\top\,,$
with $\hat \lambda_1\geq\hat \lambda_2\geq\cdots\geq \hat \lambda_{\hat r}>\hat \lambda_{\hat{r}+1}=\cdots=\hat \lambda_k=0$.  Notice that $\rank(F(\hat x))= \hat{r}\geq t$.
Following \cite{Nikolov}, we  set 
$\hat{\Theta}:=\sum_{\ell=1}^{k} {\hat \beta}_\ell \hat{u}_\ell \hat{u}_\ell^\top$\,,
where
\begin{equation}\label{defbetaa}
\hat{\beta}_\ell:=\left\{
\begin{array}{ll}
        \textstyle 1/\hat{\lambda}_\ell\,,
       &~1\leq \ell\leq \hat{\iota};\\
     1/\hat{\delta},&~\hat{\iota}<\ell\leq \hat{r};\\
     (1+\epsilon)/\hat{\delta},&~\hat{r}<\ell\leq k,
\end{array}\right.
\end{equation}
 for any $\epsilon>0$, where $\hat{\iota}$ is the unique integer defined  in Lemma \ref{Ni13} for $\lambda_\ell=\hat{\lambda}_\ell$\,, and
$
\hat \delta:=\frac{1}{t-\hat \iota}\sum_{\ell=\hat \iota+1}^{k}\hat \lambda_\ell
$\thinspace.
From Lemma \ref{Ni13}, we have that $\hat\iota<t$.   Then, 
\begin{equation*} 
\textstyle
- \sum_{\ell=1}^t \log(\hat{\beta}_{\ell})=  \sum_{\ell=1}^{\hat{\iota}} \log(\hat{\lambda}_{\ell}) + (t-\hat{\iota})\log(\hat{\delta})= \Gamma_t(F(\hat{x})).
\end{equation*}
Therefore, the minimum duality gap between $\hat x$ in the modified \ref{DDGFact} and feasible solutions
of its dual problem  of the form $(\hat\Theta,\upsilon,\nu,\pi,\tau)$,
is the optimal value of
\begin{equation}\label{mingapproba}\tag{$G(\hat\Theta)$}
\begin{array}{ll}
\min&~
    -\upsilon^\top l+ \nu^\top c   + \pi^\top b +\tau s - t\\
      \mbox{s.t.} 
&     ~ \upsilon - \nu  - A^\top \pi - \tau\mathbf{e}= - \diag(F \hat \Theta F^\top) ,\\
&~\upsilon\geq 0, ~\nu\geq 0, ~\pi\geq 0.
\end{array}
\end{equation}
We note that \ref{mingapproba} is always feasible (e.g., 
$\upsilon:=0$, $\nu:=\diag(F \hat \Theta F^\top)$, $\pi:=0$, $\tau:=0$ is a feasible solution). 

Also,   \ref{mingapproba} has a simple closed-form solution for \ref{GMESP}, that is when there are no $Ax \leq b$ constraints.
 To construct this  optimal solution to \ref{mingapproba}, 
we consider the permutation $\sigma$ of the indices in $N$, such that $\diag(F\hat\Theta F^\top)_{\sigma(1)} \geq \dots \geq \diag(F\hat\Theta F ^\top)_{\sigma(n)}$\,.
If $c_{\sigma(1)}+\sum_{i=2}^n l_{\sigma(i)}>s$, we let $\varphi:=0$; otherwise we let $\varphi:=\max\{j\in N: \sum_{i=1}^j c_{\sigma(i)} +  \sum_{i=j+1}^n  l_{\sigma(i)}\leq s\}$. 
We define $P := \{\sigma(1),\dots,\sigma(\varphi)\}$ and $Q :=\{\sigma(\varphi + 2),\dots,\sigma(n)\}$.
Then, to obtain an optimal solution of  \ref{mingapproba},   we consider its  optimality conditions 
\begin{equation}\label{kktgfact1f}
\begin{array}{l}
    \diag(F\hat \Theta F^\top)  + \upsilon - \nu - \tau\mathbf{e} = 0,~\upsilon\geq 0,~\nu\geq 0,\\
    \mathbf{e}^\top x = s,~l\leq x \leq c,\\
     - \upsilon^\top l + \nu^\top c + \tau s = \diag(F\hat \Theta F^\top)^\top x.
\end{array}
\end{equation}
We  can verify that the following solution satisfies \eqref{kktgfact1f}.
\begin{align*}
    &x^*_\ell:= \begin{cases}
        c_\ell\,,\quad&\text{for }\ell \in P;\\
        l_\ell\,,&\text{for }\ell \in Q;\\
        s-\sum_{\ell\in P}c_\ell -\sum_{\ell\in Q}l_\ell\,,&\text{for }\ell = \sigma(\varphi +1), ~ \text{if  }\varphi<n,
    \end{cases}\\ 
    &\tau^* :=  \begin{cases}
    \diag(F\hat\Theta F^\top)_{\sigma(\varphi+1)}\,,\quad &\text{if  }\varphi<n;\\
    0,&\text{otherwise},
    \end{cases}\\
    &\nu^*_\ell := \begin{cases}
         \diag(F\hat\Theta F^\top)_{\ell} - \tau^*,\quad &\text{for }\ell \in P;\\
         0,&\text{otherwise},
    \end{cases}\\
    &\upsilon^*_\ell := \begin{cases}
         \tau^*-\diag(F\hat\Theta F^\top)_{\ell}\,,\quad &\text{for } \ell \in Q;\\
         0,&\text{otherwise}.
    \end{cases}
\end{align*}

For the case in which there are no inequality constraints $Ax\leq b$, given a feasible solution $\hat x$ of \ref{DDGFact},
we define the \emph{dual solution associated to} $\hat x$ as the solution 
$(\hat{\Theta},\upsilon^*,\nu^*,\tau^*)$ constructed above.

 \begin{rem}
 Still for the case of
 no $Ax\leq b$  constraints, we note that we could replace $\diag(F \hat \Theta F^\top)$ with any vector in $\mathbb{R}^n$
 in \ref{mingapproba} and in the subsequent definition of $(\upsilon^*,\nu^*,\tau^*)$, and we could 
 even allow arbitrary $l\leq c\in\mathbb{R}^n$ in  the objective function of \ref{mingapproba}  and in the subsequent definition of $x^*$,
 and we would still have optimal primal and dual solutions for  \ref{mingapproba}.
 \end{rem}
 
 In Theorem \ref{dualoptimallcfa}, we will also prove   that, in the specific case of the subproblems solved in our B\&B algorithm to solve \ref{CGMESP}, where  $l\leq c\in\{0,1\}^n$,  the  constructed closed-form solution is an optimal solution  to the dual of the modified \ref{DDGFact}; this resolves, in the positive, Conjecture 17 from \cite{SEA_proceedings}. 
  We wish to point out and emphasize that the special cases of Theorem \ref{dualoptimallcfa} for MESP and binary D-Opt are completely new.
To prove Theorem \ref{dualoptimallcfa}, we must first
establish the key technical Lemma \ref{lem:optimal_xhat_diagf}.

\begin{lem}\label{lem:optimal_xhat_diagf}
   Let $\hat{x}$ be an optimal  solution of \ref{DDGFact},   
   where 
   $0\leq x \leq \mathbf{e}$ is replaced with $l\leq x\leq c$, with $l\leq c\in\{0,1\}^n$,  and there are no side constraints $Ax\leq b$. Let  
     $F(\hat x)= F^\top \Diag(\hat x) F =: \sum_{\ell=1}^{k} \hat\lambda_{\ell} \hat u_{\ell} \hat u_{\ell}^{\top}$ be a spectral decomposition of $F(\hat x)$. Let
       $\hat{\Theta}:=\sum_{\ell=1}^{k} {\hat \beta}_\ell \hat{u}_\ell \hat{u}_\ell^\top$\,, where $\hat\beta$ is defined in \eqref{defbetaa}. Let  $\bar N := \{i\in N: l_{i}=0, c_{i} = 1\}$.
Then, for every $i,j\in \bar{N}$, 
we have
\begin{enumerate}
    \item[\rm(a)] \label{item1} $\diag(F\hat\Theta F^\top)_{i} \geq \diag(F\hat\Theta F^\top)_{j}\,$, if $\hat{x}_i > \hat{x}_j$\,,
    \item[\rm(b)] \label{item2} $\diag(F\hat\Theta F^\top)_{i} = \diag(F\hat\Theta F^\top)_{j}\,$, if $\hat{x}_i\,, \hat{x}_j\in(0,1)$.
    \end{enumerate}
\end{lem}

\begin{proof}
Let $\tilde{x}$ be a feasible solution to the ``\emph{modified} \ref{DDGFact}'' (refereed to below simply as \ref{DDGFact}), such that $F(\tilde{x})$ is in the domain of $\Gamma_t$\,, i.e., $F(\tilde{x})\succeq 0$ with $\rank(F(\tilde{x})) \geq t$. From \cite[Proposition 11]{ChenFampaLeeGenScaling}, we have that  the directional derivative of $\Gamma_t$ at $\hat{x}$ in the direction $\tilde{x}-\hat{x}$ exists, and is given by 
    \[
    (\tilde{x}-\hat{x})^\top\textstyle\frac{\partial  \Gamma_t(F(\hat x))}{\partial  x} = (\tilde{x}-\hat{x})^\top\diag(F\hat\Theta F^\top).
    \]
Then,  because  \ref{DDGFact} is a convex optimization problem with a concave objective function $\Gamma_t$\,, we  conclude that $\hat{x}$ is an optimal solution to \ref{DDGFact} if and only if  $(\tilde{x}-\hat{x})^\top\diag(F\hat\Theta F^\top) \leq 0$ for every feasible solution $\tilde{x}$ to \ref{DDGFact}, such that $F(\tilde{x})$ is in the domain of $\Gamma_t$\,. 

\begin{enumerate}[wide, labelwidth=0pt, labelindent=0pt]
\item[\rm(a)] 
    Suppose  there exist  $\hat{\imath},\hat{\jmath}\in\bar{N}$, such that $\hat{x}_{\hat\imath} > \hat{x}_{\hat\jmath}$ and 
    $\diag(F\hat\Theta F^\top)_{\hat{\imath}} < \diag(F\hat\Theta F^\top)_{\hat{\jmath}}\,$. Let $0<\epsilon<\hat{x}_{\hat\imath} - \hat{x}_{\hat\jmath}$\,. Let $\tilde{x}$ be the feasible solution to \ref{DDGFact}, such that $\tilde{x}_\ell := \hat{x}_\ell$ for $\ell \in \bar{N}\setminus\{\hat{\imath},\hat{\jmath}\}$, $\tilde{x}_{\hat{\imath}} := \hat{x}_{\hat{\jmath}} +\epsilon$ and $\tilde{x}_{\hat{\jmath}} := \hat{x}_{\hat{\imath}} - \epsilon$. Note that $\tilde{x}_{\hat{\imath}}$ and $\tilde{x}_{\hat{\jmath}}$ are both positive, so $\rank(F(\tilde{x}))\geq\rank(F(\hat{x}))\geq t$, and $F(\tilde{x})$ is in the domain of $\Gamma_t$\,. We have
    \begin{align*}
        (\tilde{x}-\hat{x})^\top(\diag(F\hat\Theta F^\top))  &=(\tilde{x}_{\hat\imath}-\hat{x}_{\hat\imath})\diag(F\hat\Theta F^\top)_{\hat{\imath}} + (\tilde{x}_{\hat\jmath}-\hat{x}_{\hat\jmath})\diag(F\hat\Theta F^\top)_{\hat{\jmath}} \\
        &=(\hat{x}_{\hat\jmath}-\hat{x}_{\hat\imath} +\epsilon)(\diag(F\hat\Theta F^\top)_{\hat{\imath}}-\diag(F\hat\Theta F^\top)_{\hat{\jmath}})> 0,
    \end{align*}
    which contradicts the optimality of $\hat{x}$. The result in part (a) follows.
\medskip
    \item[\rm(b)]
    Suppose  there exist   $\hat{\imath},\hat{\jmath}\in\bar{N}$, such that $\hat{x}_{\hat\imath}, \hat{x}_{\hat\jmath}\in(0,1)$. Assume, without loss of generality that $\hat{x}_{\hat\imath}> \hat{x}_{\hat\jmath}$\,. From part (a), we have $\diag(F\hat\Theta F^\top)_{\hat{\imath}} \geq \diag(F\hat\Theta F^\top)_{\hat{\jmath}}$\,. Now, suppose  
    $\diag(F\hat\Theta F^\top)_{\hat{\imath}} > \diag(F\hat\Theta F^\top)_{\hat{\jmath}}$\,. Let $0<\epsilon<\min\{1-\hat{x}_{\hat{\imath}},\hat{x}_{\hat{\jmath}}\}$. Let $\tilde{x}$ be the  feasible solution to \ref{DDGFact} such that 
 $\tilde{x}_\ell := \hat{x}_\ell$ for $\ell \in \bar{N}\setminus\{\hat{\imath},\hat{\jmath}\}$, $\tilde{x}_{\hat{\imath}} := \hat{x}_{\hat{\imath}} + \epsilon$ and $\tilde{x}_{\hat{\jmath}} := \hat{x}_{\hat{\jmath}} - \epsilon$. As in part (a), $F(\tilde{x})$ is in the domain of $\Gamma_t$\,.
     We have
    \begin{align*}
        (\tilde{x}-\hat{x})^\top(\diag(F\hat\Theta F^\top)) 
        &=\epsilon(\diag(F\hat\Theta F^\top)_{\hat{\imath}}-\diag(F\hat\Theta F^\top)_{\hat{\jmath}})> 0,
    \end{align*}
    which contradicts the optimality of $\hat{x}$. The result in part (b) follows.  \qed  
    \end{enumerate} 
\end{proof}

\begin{thm}\label{dualoptimallcfa}
   Let $\hat{x}$ be an optimal  solution of \ref{DDGFact}, where 
   $0\leq x \leq \mathbf{e}$ is replaced with $l\leq x\leq c$, with $l\leq c\in\{0,1\}^n$, and there are no side constraints $Ax\leq b$. Then,  the dual solution $(\hat{\Theta},\upsilon^*,\nu^*,\tau^*)$ associated to $\hat{x}$ is an optimal solution  to its dual problem. 
\end{thm}
\begin{proof} 
We must prove that
$-{\upsilon^*}^\top l + {\nu^*}^\top c    +\tau^* s - t\! =\! 0$. 
 So, it suffices to show~that
    \begin{equation}\label{toprovea}
    \textstyle\sum_{i \in P}  \left(\diag(F \hat\Theta F^\top)_{i} - \tau^*\right)c_{i} + \sum_{i \in Q}  \left(\diag(F \hat\Theta F^\top)_{i} -\tau^*\right)l_{i}  = t- \tau^* s\,,\end{equation}
 
 We note that in the case where $l\leq c\in\{0,1\}^n$, we can verify from the definition of $\varphi$ that  
$\textstyle\sum_{i\in P} c_{i} +  \sum_{i\in Q}  l_{i}=s$; and  $l_{\sigma(\varphi+1)}=0$, whenever $\varphi<n$. Then, we have $\tau^\star\textstyle\sum_{i\in P} c_{i} +  \tau^\star\sum_{i\in Q}  l_{i}=\tau^\star s$. 

We consider the partition of $N$ into three subsets:  $U_0 := \{i\in N: l_{i}=c_{i} = 0\}$, $L_1 := \{i\in N: l_{i} =c_{i}= 1\}$, and  $\bar N := N\setminus(L_1\cup U_0)=\{i\in N: l_{i}=0, c_{i} = 1\}$. We  note that  $L_1\subseteq P\cup Q$ because $l_{\sigma(\varphi+1)}=0$, whenever $\varphi<n$. Then,
       \begin{align*}
           &\textstyle\sum_{i \in P}  \diag(F \hat\Theta F^\top)_{i} c_{i} + \textstyle\sum_{i \in Q}  \diag(F \hat\Theta F^\top)_{i} l_{i}\\ 
           &=\textstyle\sum_{i \in P\cap\bar{N} }\!\diag(F \hat\Theta F^\top)_{i} c_{i} \!+\! \sum_{i \in P\cap L_1}\! \!\!\diag(F \hat\Theta F^\top)_{i} c_{i}  \!+\!  \sum_{i\in Q \cap L_1}\!\!\!\diag(F \hat\Theta F^\top)_{i} l_{i}
           \\           
           &=\textstyle\sum_{i \in P\cap\bar{N}} \diag(F \hat\Theta F^\top)_{i} + \sum_{i \in L_1}  \diag(F \hat\Theta F^\top)_{i}\,.
       \end{align*}
       Also, from \eqref{fdottheta}, we have that $\hat{x}^\top \diag(F \hat\Theta F^\top)=t$. Then, to show \eqref{toprovea}, it suffices to show that 
       \begin{equation}\label{toprove}
       \textstyle\sum_{i \in P\cap\bar{N}}  \diag(F \hat\Theta F^\top)_{i} + \sum_{i \in L_1}  \diag(F \hat\Theta F^\top)_{i}= \hat{x}^\top \diag(F \hat\Theta F^\top).
       \end{equation}
       We can also  verify that 
       \[
          s= \textstyle\sum_{i\in P} c_{i} +  \sum_{i\in Q}  l_{i}
          = \textstyle\sum_{i \in P\cap\bar{N} } 1 +   \sum_{i\in L_1}1=|P\cap\bar{N}| + |L_1|.
       \]
       Let
       $\mathcal{I}_1:=\{i\in \bar N: \hat{x}_{i}=1\}$. If   $\hat x \in \{0,1\}^n$,  then $|\mathcal{I}_1| = s - |L_1|$ and $|P\cap \bar{N}| = |\mathcal{I}_1|$.  So, from Lemma \ref{lem:optimal_xhat_diagf}, part (a),  and the ordering defined by $\sigma$, we have that $P\cap \bar{N} = \mathcal{I}_1$\,, and \eqref{toprove}  follows.
        Next, suppose that $\hat{x} \notin \{0,1\}^n$. Let $\mathcal{I}_f:=\{i\in \bar{N}: \hat{x}_{i}\in(0,1)\}$. Note that $\sum_{i\in\mathcal{I}_f}\hat{x}_i=s-|\mathcal{I}_1|-|L_1|$. 
    Let $\hat{d}:= \diag(F\hat{\Theta}F^\top)_{i}$\,, for every $i \in \mathcal{I}_f$ (this is well defined, due to  Lemma \ref{lem:optimal_xhat_diagf}, part (b)).
    Then, 
   \[
   \hat{x}^\top \diag(F \hat\Theta F^\top)
        =  \textstyle\sum_{i \in \mathcal{I}_1}  \diag(F \hat\Theta F^\top)_{i} + (s-|\mathcal{I}_1| - |L_1|) \hat{d} + \textstyle\sum_{i \in L_1}  \diag(F \hat\Theta F^\top)_{i}\,.
\]
 Note that  $|\mathcal{I}_f|>s-|\mathcal{I}_1| - |L_1|$, so $|P\cap\bar{N}|<|\mathcal{I}_1| + |\mathcal{I}_f|$. Also,   from Lemma \ref{lem:optimal_xhat_diagf}, part (a), we see  that $\diag(F\hat{\Theta}F^\top)_i \geq \hat{d}$ for all $i \in \mathcal{I}_1$\,. Then, we have
   \begin{align*}
          &\textstyle\sum_{i \in P\cap\bar{N}}  \diag(F \hat\Theta F^\top)_{i} + \sum_{i \in L_1}  \diag(F \hat\Theta F^\top)_{i} \\                   
            &\quad = \textstyle\sum_{i \in \mathcal{I}_1} \diag(F \hat\Theta F^\top)_{i} + (s-|\mathcal{I}_1| - |L_1|)\hat{d} + \sum_{i \in L_1}  \diag(F \hat\Theta F^\top)_{i}~.
       \end{align*}   
       The result follows.
     \qed
\end{proof}


\section{Experiments}\label{sec:exp}

We will discuss experiments that we carried out for \ref{GMESP} and \ref{CGMESP}. For all instances, we use a benchmark positive-definite covariance matrix $C$ of dimension $n=63$,  originally  
obtained from J. Zidek (University of British Columbia), coming from an application for re-designing an environmental monitoring network;
see \cite{Guttorp-Le-Sampson-Zidek1993} and \cite{HLW}. This  matrix has been used extensively in testing and developing algorithms for MESP; see \cite{KLQ,LeeConstrained,AFLW_Using,LeeWilliamsILP,HLW,AnstreicherLee_Masked,BurerLee,Anstreicher_BQP_entropy,Kurt_linx,Mixing,ChenFampaLee_Fact}.   For testing \ref{CGMESP}, we utilized ten  side constraints  $a_i^\top x\leq b_i$, for $i=1,\ldots,10$.  We have generated them randomly.  The  left-hand side of constraint $i$ is given by a uniformly-distributed random vector $a_i$ with integer components between $0$ and $5$. The right-hand side of the constraints is selected so that, for
every $s$ considered in the experiment, the best known solution 
of the instance of \ref{GMESP} is violated by
at least one constraint. 

We ran our experiments on ``zebratoo", a
32-core machine (running Windows Server 2022 Standard):
two Intel Xeon Gold 6444Y processors running at 3.60GHz, with 16 cores each, and 128 GB of memory.
We coded our  algorithms   in \texttt{Julia} v.1.9.0. To solve the convex relaxation \ref{DDGFact}, we used \texttt{Knitro}  v0.13.2, using \texttt{CONVEX = true},
\texttt{FEASTOL} = $10^{-8}$,
\texttt{OPTTOL} = $10^{-8}$, \texttt{ALGORITHM = 1} (Interior/Direct algorithm), 
\texttt{HESSOPT = 3} ((dense) Quasi-Newton SR1).\footnote{\texttt{HESSOPT = 2} ((dense) BFGS) might seem like the more natural choice for a convex problem, but we observed better behavior with the symmetric rank-one update, which is not without precedent; see \cite{byrd}.}
To solve \ref{GMESP} and \ref{CGMESP}, we adapted  the B\&B algorithm in  \texttt{Juniper} (Jump Non linear Integer Program solver) \cite{juniper} (using the \texttt{MostInfeasible}  branching rule,
and $10^{-5}$ as the tolerance to consider a value as integer). For the spectral bound for \ref{CGMESP}, we used \texttt{LBFGSB.jl}, a  \texttt{Julia} wrapper for the well-known \texttt{L-BFGS-B code} (the limited-memory bound-constrained BFGS algorithm; see \cite{LBFGSB}), using \texttt{factr} = $10^7$, \texttt{pgtol} = $10^{-5}$, \texttt{maxfun} = $15000$, \texttt{maxiter} = $15000$. We note that even though the calculation of the  spectral bound for \ref{CGMESP} requires solving a non-smooth optimization problem, it is well known that BFGS approaches often work quite well, regardless
(see \cite{nonsmoothBFGS}, and in particular \cite[Sec. 6]{LeeConstrained}). 

\subsection{Lower bounds}\label{sec:LB}

To get an idea of the performance of upper bounds,
we present gaps to good lower bounds.
For good lower bounds for \ref{GMESP}, we carried out an appropriate local search, in the spirit of \cite[Sec. 4]{KLQ}, starting 
from various good feasible solutions. 
Our local search is classical: Starting from some $S$ with $|S|=s$, 
we iteratively replace $S$ with $S+j-i$
when $\prod_{\ell=1}^t \lambda_{\ell} (C[S+j-i,S+j-i])
> \prod_{\ell=1}^t \lambda_{\ell} (C[S,S])$. 
We return at the end $x(S)$, the characteristic vector of $S$. 

We have three methods for generating initial solutions for the local search. 
\begin{itemize}
\item \emph{Rounding a continuous solution.}
Let $\lambda_\ell(C)$ be the $\ell$-th greatest eigenvalue of $C$, and  $u_\ell$ be the corresponding eigenvector, normalized to have Euclidean length 1.
We define $\bar{x}\in \mathbb{R}^n$
by $\bar{x}_j:= \sum_{\ell=1}^t u_{\ell j}^2$\,. It is easy to check 
(similar to \cite[Sec. 3]{LeeConstrained})
that $0\leq \bar{x}\leq \mathbf{e}$ and $\mathbf{e}^\top \bar{x}= t$. 
Next, we simply choose $S$ to comprise the indices $j$ corresponding to the $s$ 
biggest $\bar{x}_j$\,;
we note that this rounding method can be
adapted to \ref{CGMESP}, by instead solving
a small integer linear optimization problem (see \cite[Sec. 4]{LeeConstrained}).
\item \emph{Greedy.}
Starting from $S:=\emptyset$,
we identify the element $j\in N\setminus S$ that maximizes the product of the 
$\min \{t, |S|+1\}$ 
greatest eigenvalues of $C[S+j,S+j]$. 
We let $S:=S+j$, and we repeat while $|S|<s$.
\item \emph{Dual greedy.}
Starting from $S:=N$,
we identify the element $j\in S$ that maximizes the product of the 
$t$ 
greatest eigenvalues of $C[S-j,S-j]$. 
We let $S:=S-j$, and we repeat while $|S|>s$.
\end{itemize} 

To generate  lower bounds for \ref{CGMESP}, we  apply a rounding heuristic, using a similar idea to  \cite[Sec. 4]{LeeConstrained}, but considering the continuous solution obtained by solving  \ref{DDGFact}. More specifically, we obtain the optimal solution $\hat{x}$ of \ref{DDGFact}, and then solve the  integer linear-optimization problem $\max\{\hat{x}^\top x: \mathbf{e}^\top x=s, Ax\leq b, x\in\{0,1\}^n\}$.

\subsection{Behavior of upper bounds}
 
To analyze the generalized factorization bound for \ref{GMESP} and \ref{CGMESP} and compare it to the spectral bound, we conducted two  experiments using our covariance matrix with $n=63$, and ten randomly-generated side constraints for \ref{CGMESP}. 

In the first experiment, for each integer $\kappa$ 
from $0$ to $3$,  we consider the  instances obtained when we vary $s$ from $\kappa +1$ to $61$, and set $t:=s\!-\!\kappa$. 

In  Figures \ref{fig:varying_k} and \ref{fig:varying_kc}, we depict for each $\kappa$,  the gaps given by the difference between the lower bounds for \ref{GMESP} and \ref{CGMESP}, respectively,  computed as described in \S\ref{sec:LB}, and the   generalized factorization bound (``$\Gamma$") and the spectral bound (``$\mathcal{S}$") for each pair $(s,t:=s-\kappa)$. 
 These so-called ``absolute gaps'' are presented, as is typical in the literature for such problems having a log objective function, which does not always take on positive values.
In Figure \ref{fig:varying_k}, we also depict  the upper bound on the 
gap for the generalized factorization bound  (``ub on $\Gamma$") determined in Theorem \ref{thm:almost_dominates}, specifically given by the gap for the spectral bound added to $t\log(s/t)$. When $\kappa=0$ (instances of MESP),  $t\log(s/t)$ is  zero,  confirming that the generalized factorization bound dominates the spectral bound, as established in \cite{ChenFampaLee_Fact}. 

Observing the plots in Figure \ref{fig:varying_k}, we see that when $\kappa$ increases,  the generalized factorization bound becomes weaker and gets worse than the spectral bound when $s$ and $t$ get large enough.   Nevertheless, the generalized factorization bound is still much stronger than the spectral bound for most of the instances considered, and we see that the upper bound on it given in Theorem \ref{thm:almost_dominates} is, in general,  very loose, especially  for $s\in[10,20]$. The variation of the gaps for generalized factorization bounds for the different values of $s$ is very small when compared to the spectral bound, showing its robustness. Our previous experience for MESP and now 
for \ref{GMESP} (see Table \ref{tab:BB}) has been that
root gaps of less than one (resp., greater than two) indicate that the 
instance is likely solvable (resp., not solvable) by B\&B
in reasonable time.

\begin{figure}[!ht]
\captionsetup[subfigure]{aboveskip=0pt,belowskip=8pt}
\centering
\begin{subfigure}{0.49\textwidth}
    \includegraphics[width=\textwidth]{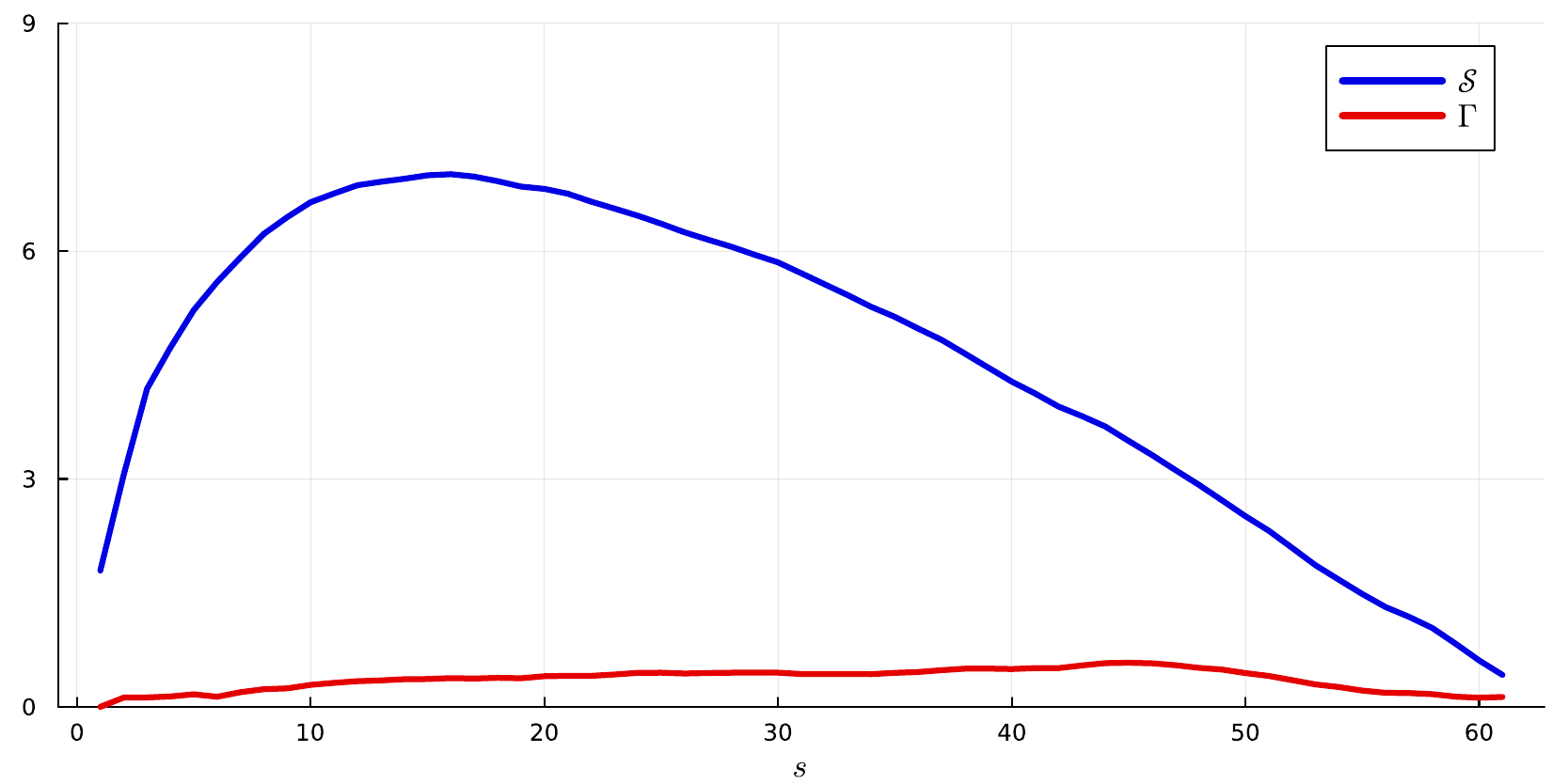}
    \caption{$\kappa=0$}
    \label{fig:k0}
\end{subfigure}
\hfill
\begin{subfigure}{0.49\textwidth}
    \includegraphics[width=\textwidth]{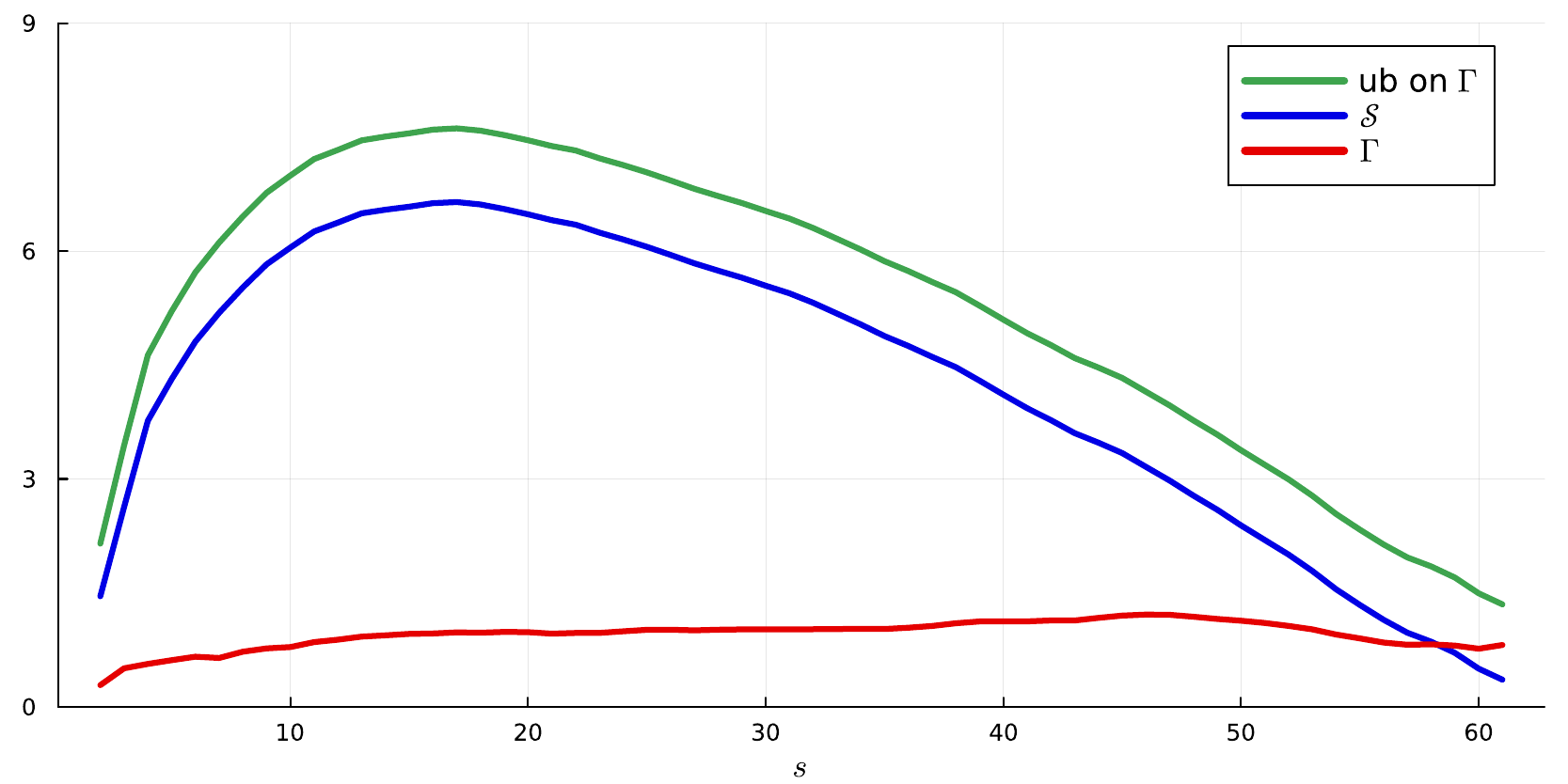}
    \caption{$\kappa=1$}
    \label{fig:k1}
\end{subfigure}
\hfill
\begin{subfigure}{0.49\textwidth}
    \includegraphics[width=\textwidth]{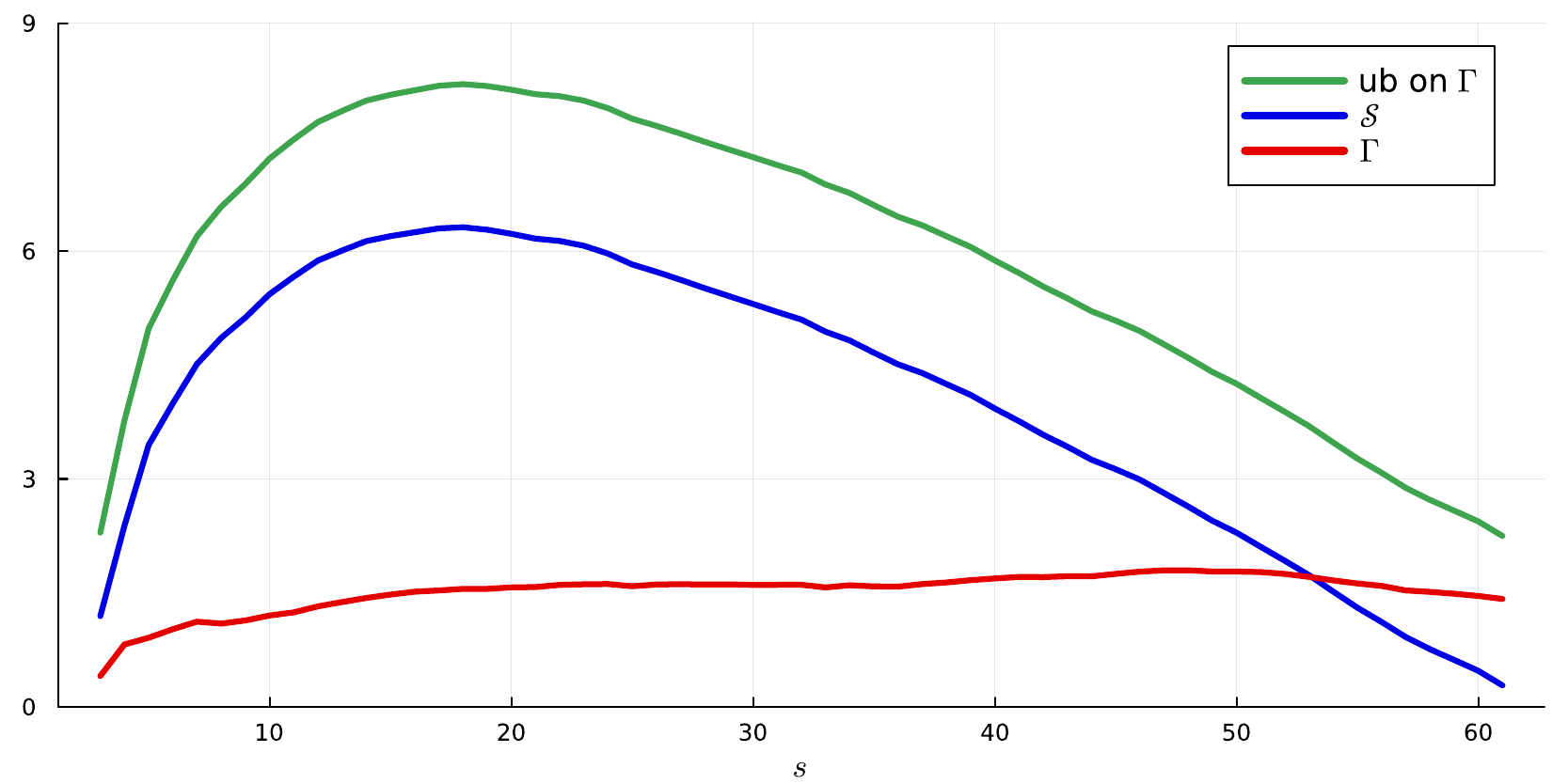}
    \caption{$\kappa=2$}
    \label{fig:k2}
\end{subfigure}
\hfill
\begin{subfigure}{0.49\textwidth}
    \includegraphics[width=\textwidth]{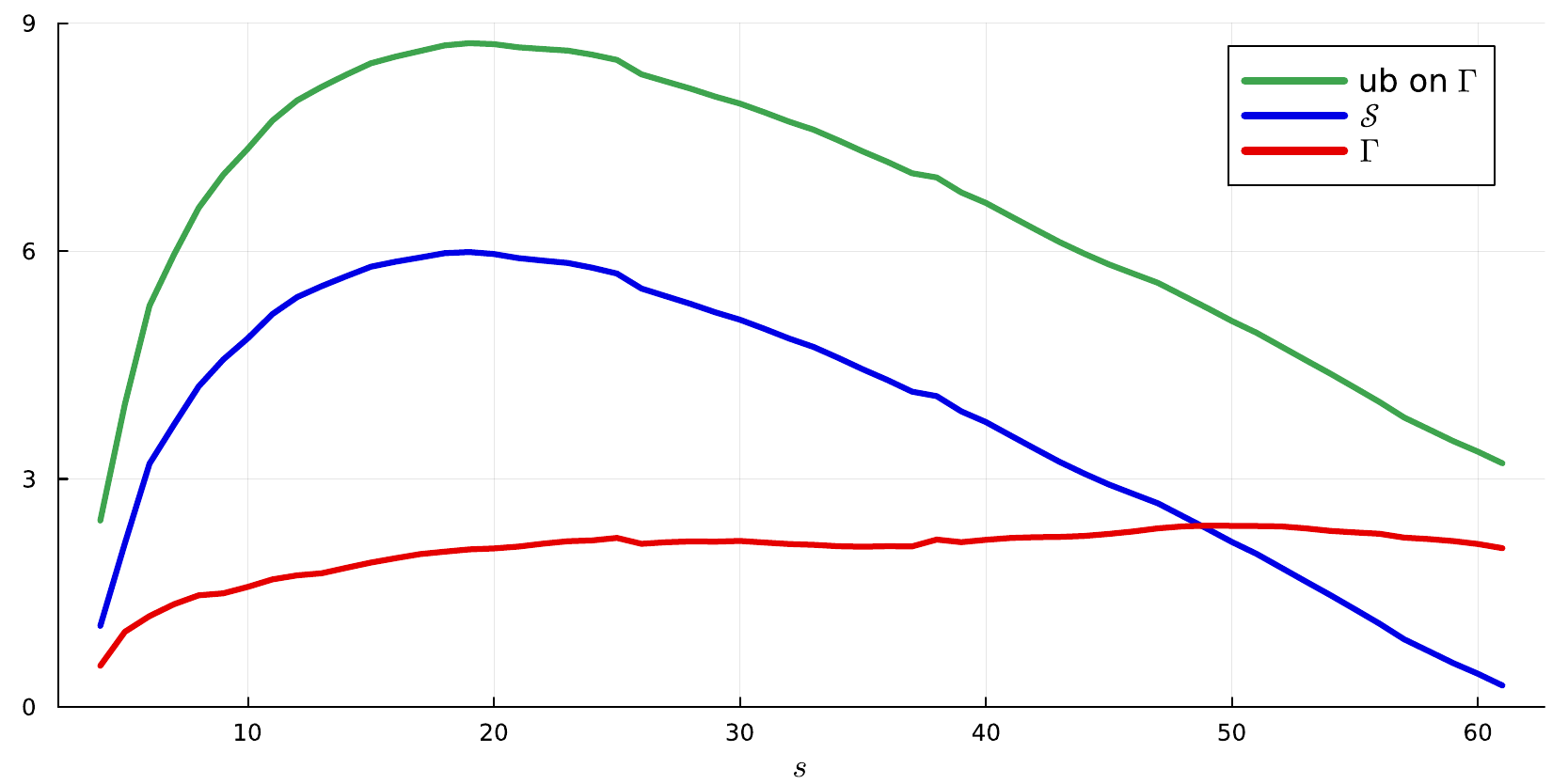}
    \caption{$\kappa=3$}
    \label{fig:k3}
\end{subfigure}
        
\caption{Gaps for \ref{GMESP},  varying $t=s-\kappa$ ($n=63$)}
\label{fig:varying_k}
\end{figure}

In Figure \ref{fig:varying_kc}, we observe that the gaps
for the generalized factorization bound are generally much worse
than for unconstrained instances. For the spectral bound, this is
not generally so pronounced. But the 
spectral-bound gaps do not tend to zero, as $s$ tends toward $n$,
like they do for the unconstrained instances. Because of this,
we do not see that there is a threshold for $s$,
beyond which the spectral bound dominates the 
generalized factorization bound. 
We can also observe that the
behavior of the bounds for constrained instances, as we vary $s$, is much more irregular than for unconstrained instances; this
is due to the fact that the constraints are changing, as we vary $s$, for constrained instances, and also probably because some of  the heuristically-generated lower bounds are quite poor. 

\begin{figure}[!ht]
\captionsetup[subfigure]{aboveskip=0pt,belowskip=8pt}
\centering
\begin{subfigure}{0.49\textwidth}
    \includegraphics[width=\textwidth]{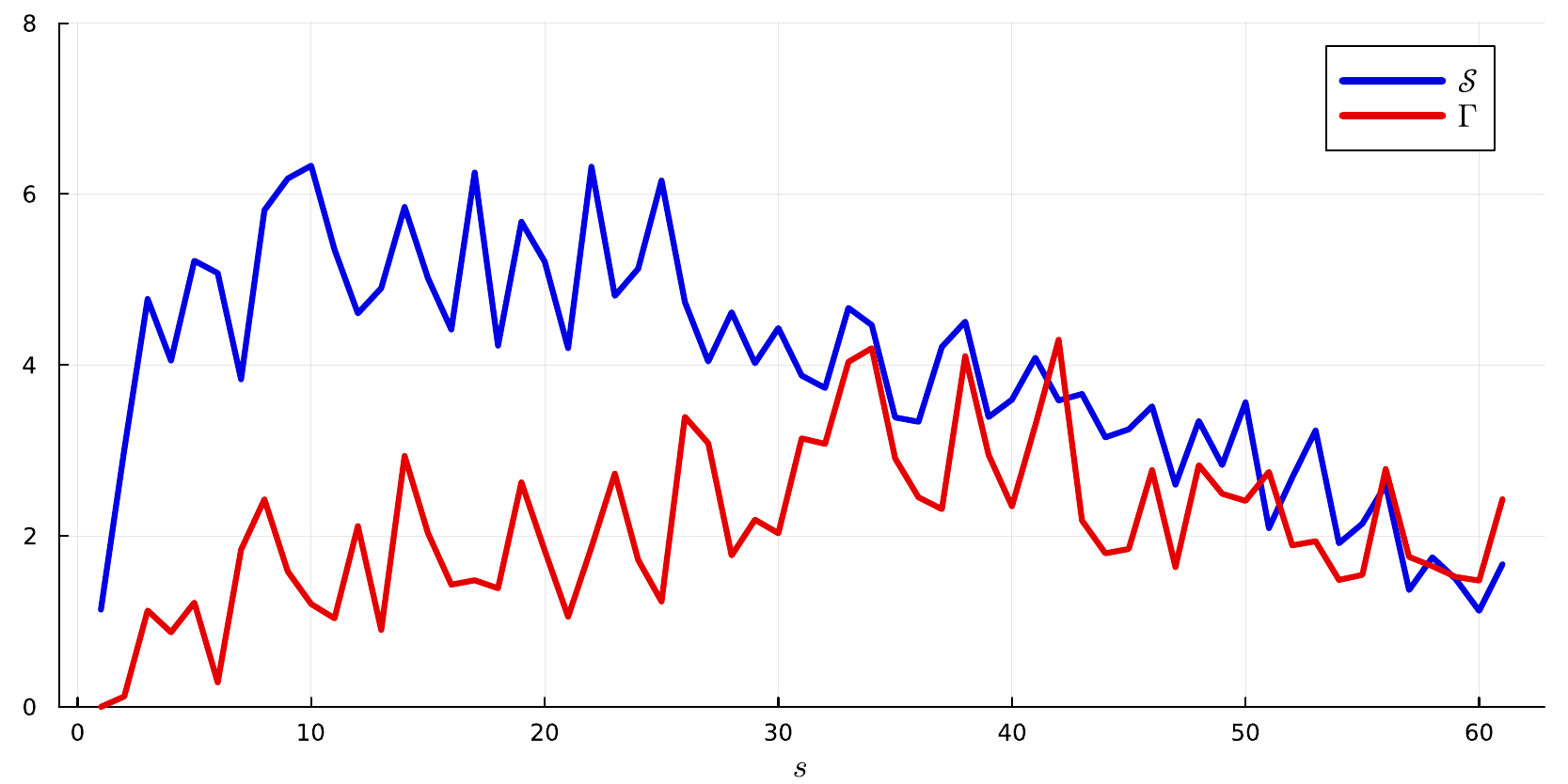}
    \caption{$\kappa=0$}
    \label{fig:k0c}
\end{subfigure}
\hfill
\begin{subfigure}{0.49\textwidth}
    \includegraphics[width=\textwidth]{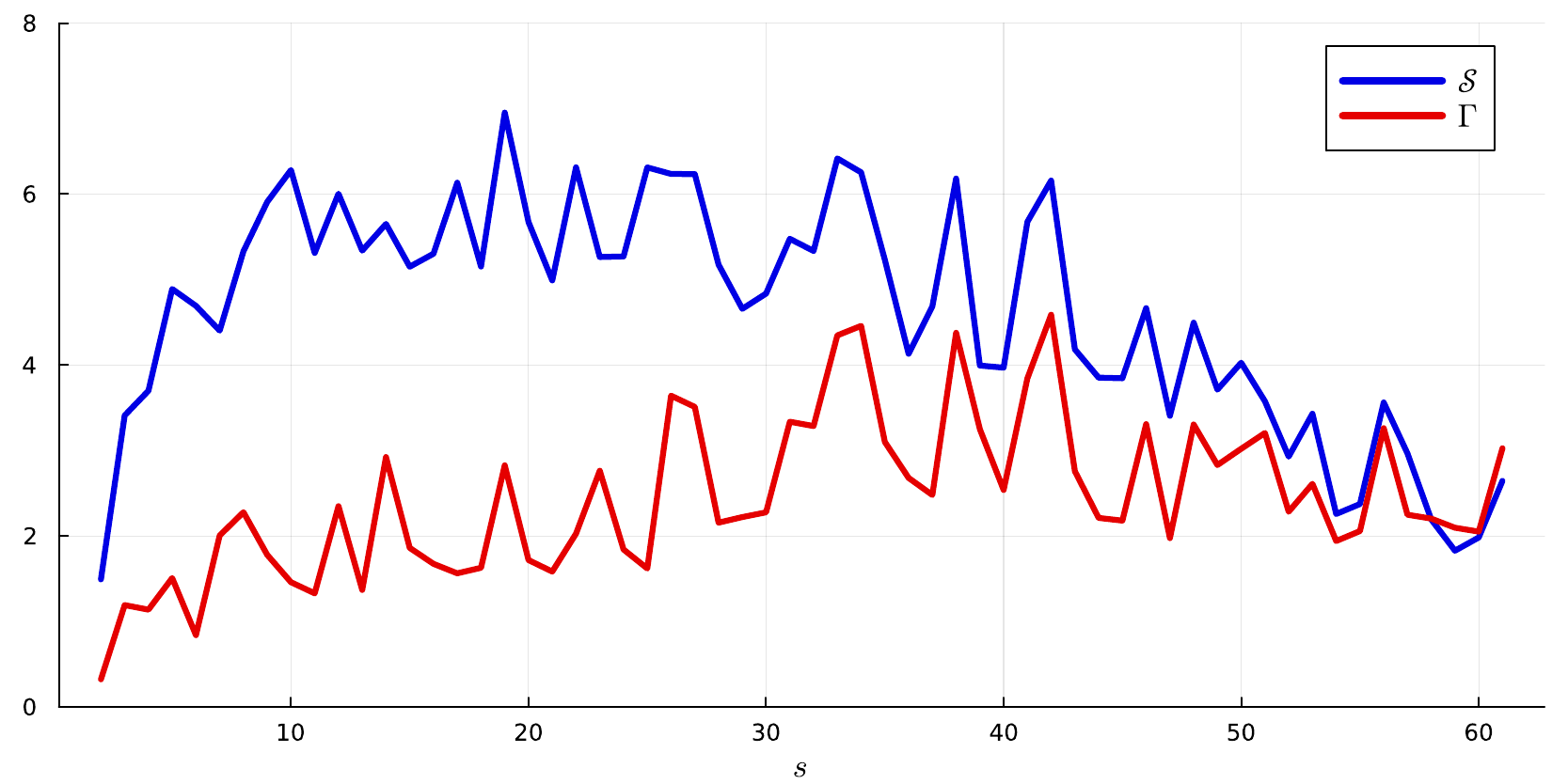}
    \caption{$\kappa=1$}
    \label{fig:k1c}
\end{subfigure}
\hfill
\begin{subfigure}{0.49\textwidth}
    \includegraphics[width=\textwidth]{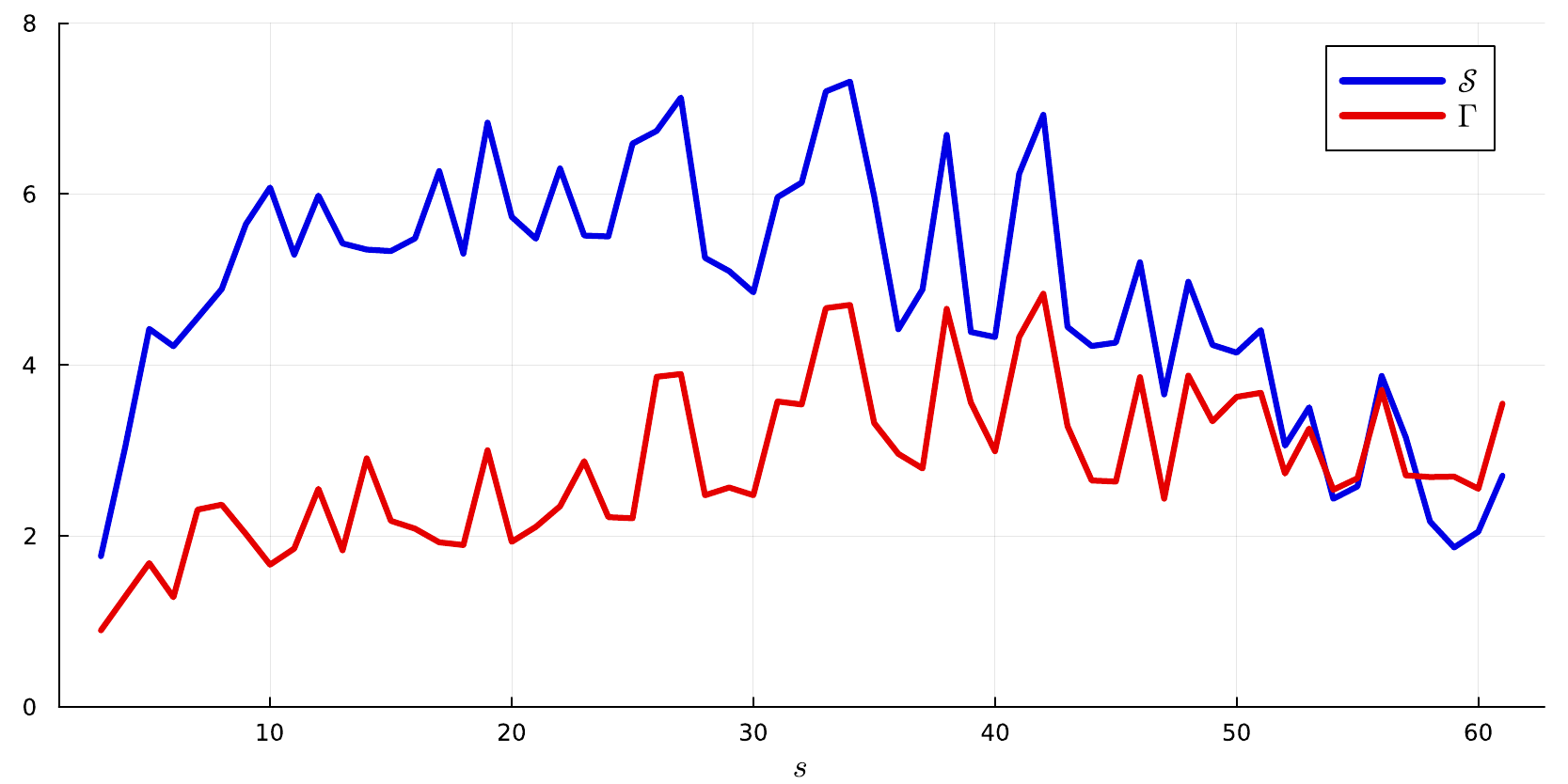}
    \caption{$\kappa=2$}
    \label{fig:k2c}
\end{subfigure}
\hfill
\begin{subfigure}{0.49\textwidth}
    \includegraphics[width=\textwidth]{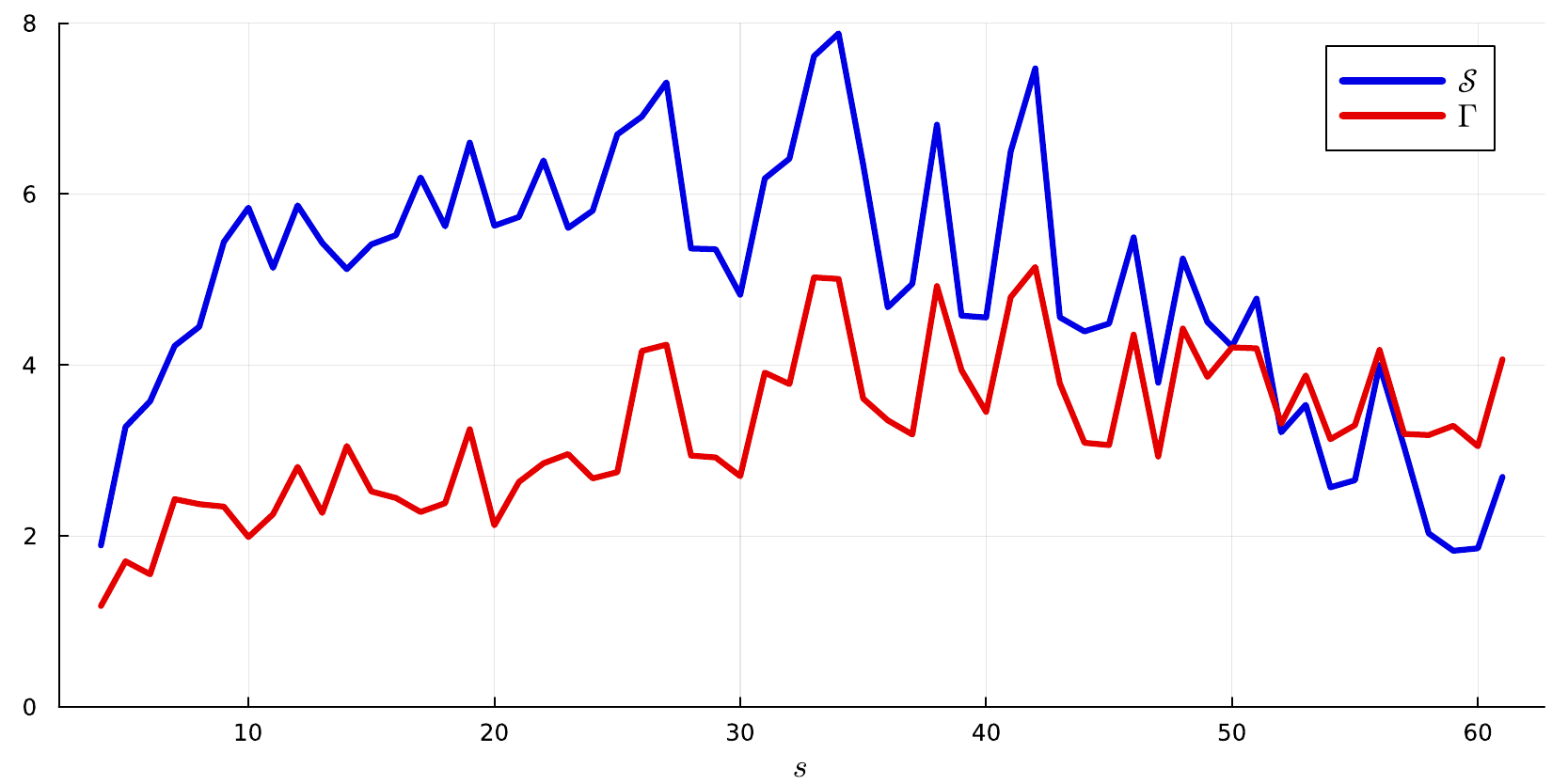}
    \caption{$\kappa=3$}
    \label{fig:k3c}
\end{subfigure}
        
\caption{Gaps for \ref{CGMESP},  varying $t=s-\kappa$ ($n=63$, $m=10$)}
\label{fig:varying_kc}
\end{figure}

In the second experiment, for $s=10, 20, 40$,   we consider the  instances obtained when we vary $t$ from $1$ to $s$. 
Similarly to what we show in Figures \ref{fig:varying_k} and \ref{fig:varying_kc}, in  Figure \ref{fig:varying_t}, we show the gaps  for each instance, and the upper bound on the gap corresponding to the generalized factorization bound for \ref{GMESP}.

\begin{figure}[!ht]
\captionsetup[subfigure]{aboveskip=0pt,belowskip=8pt}
\centering
\begin{subfigure}{0.49\textwidth}
    \includegraphics[width=\textwidth]{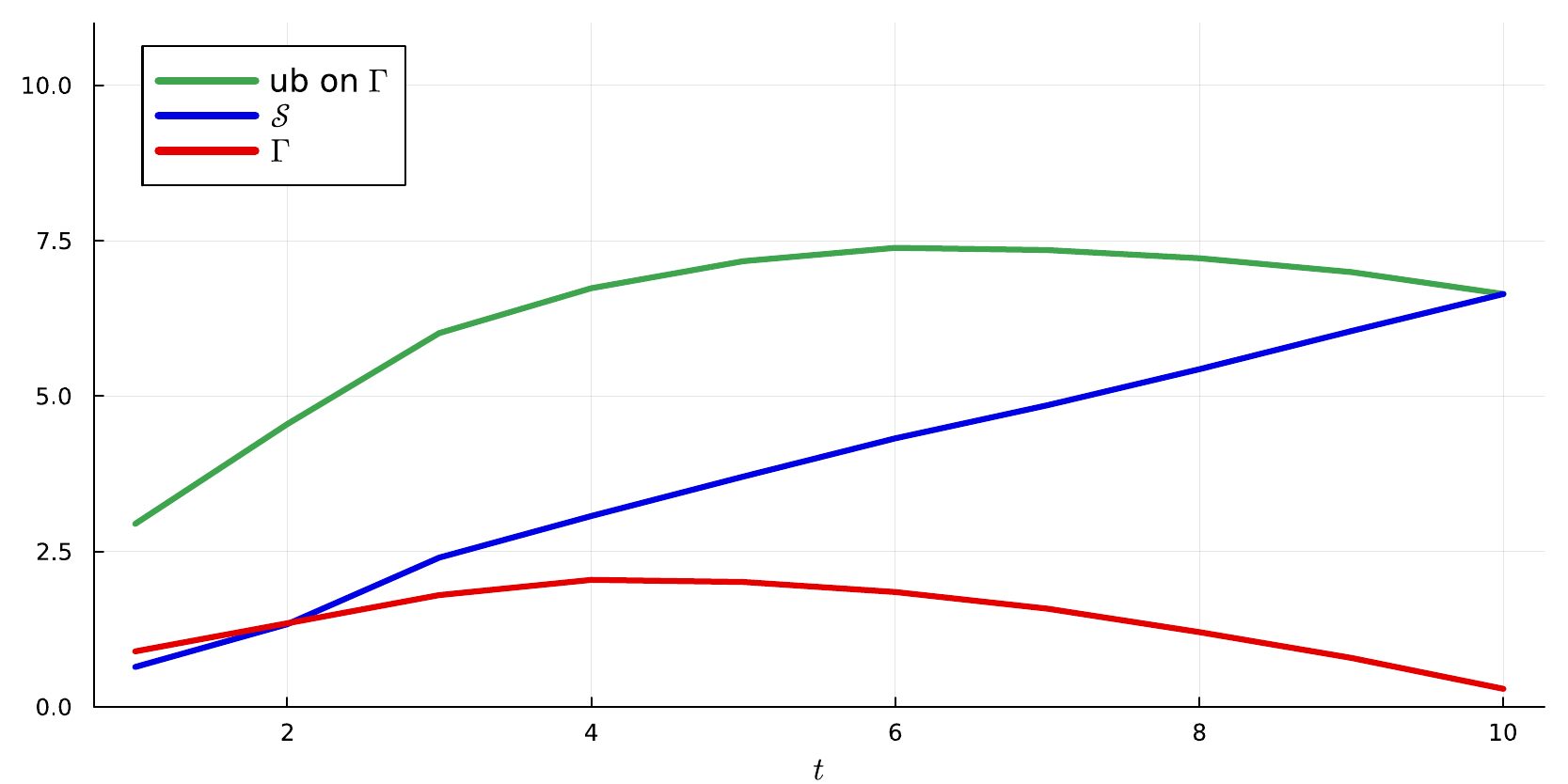}
    \vspace{-0.1in}
    \caption{\ref{GMESP}, $s=10$}
    \label{fig:s10}
\end{subfigure}
\hfill
\begin{subfigure}{0.49\textwidth}
    \includegraphics[width=\textwidth]{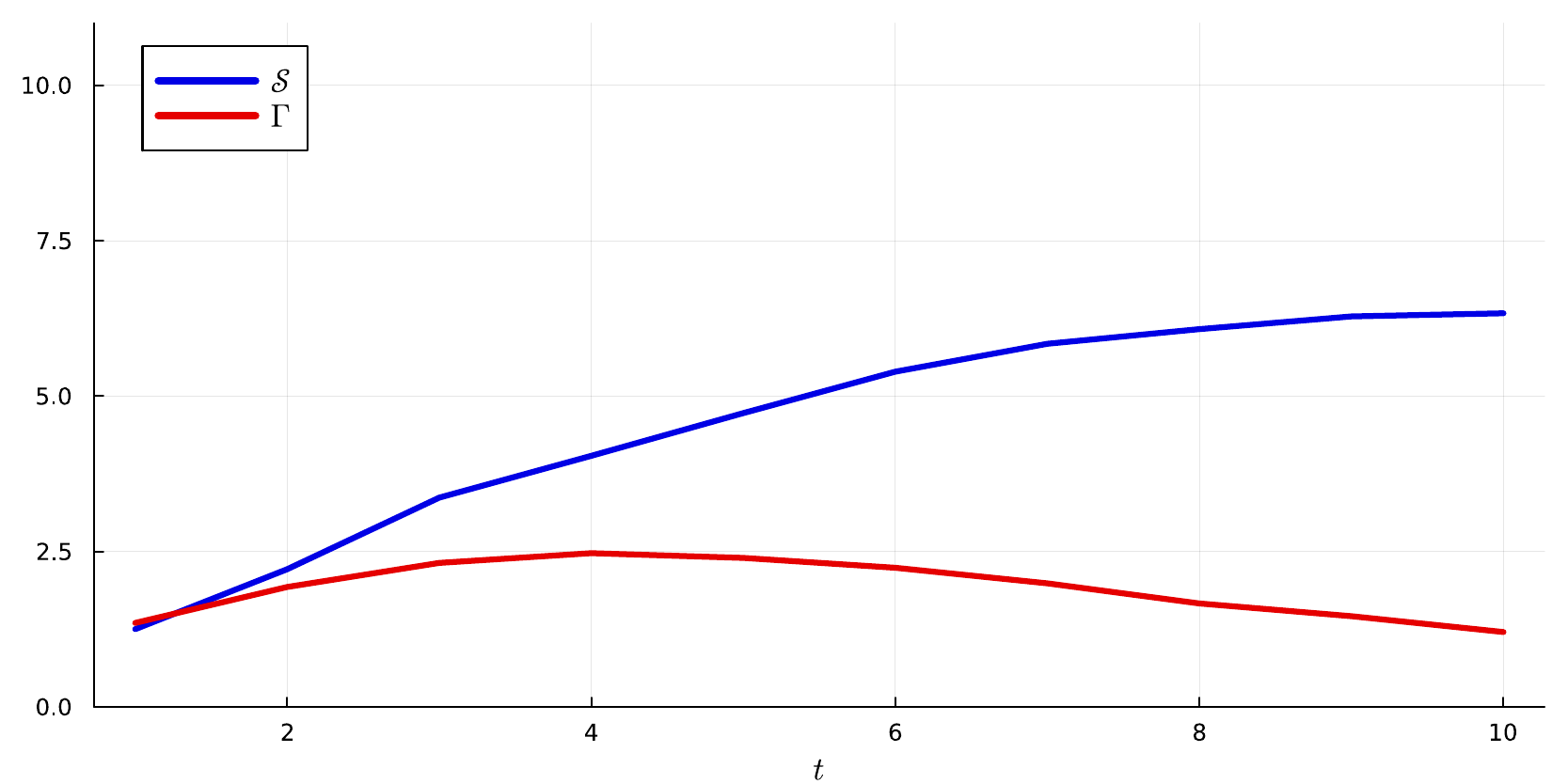}
    \caption{\ref{CGMESP}, $m=10$, $s=10$}
    \label{fig:s10c}
\end{subfigure}

\begin{subfigure}{0.49\textwidth}
    \includegraphics[width=\textwidth]{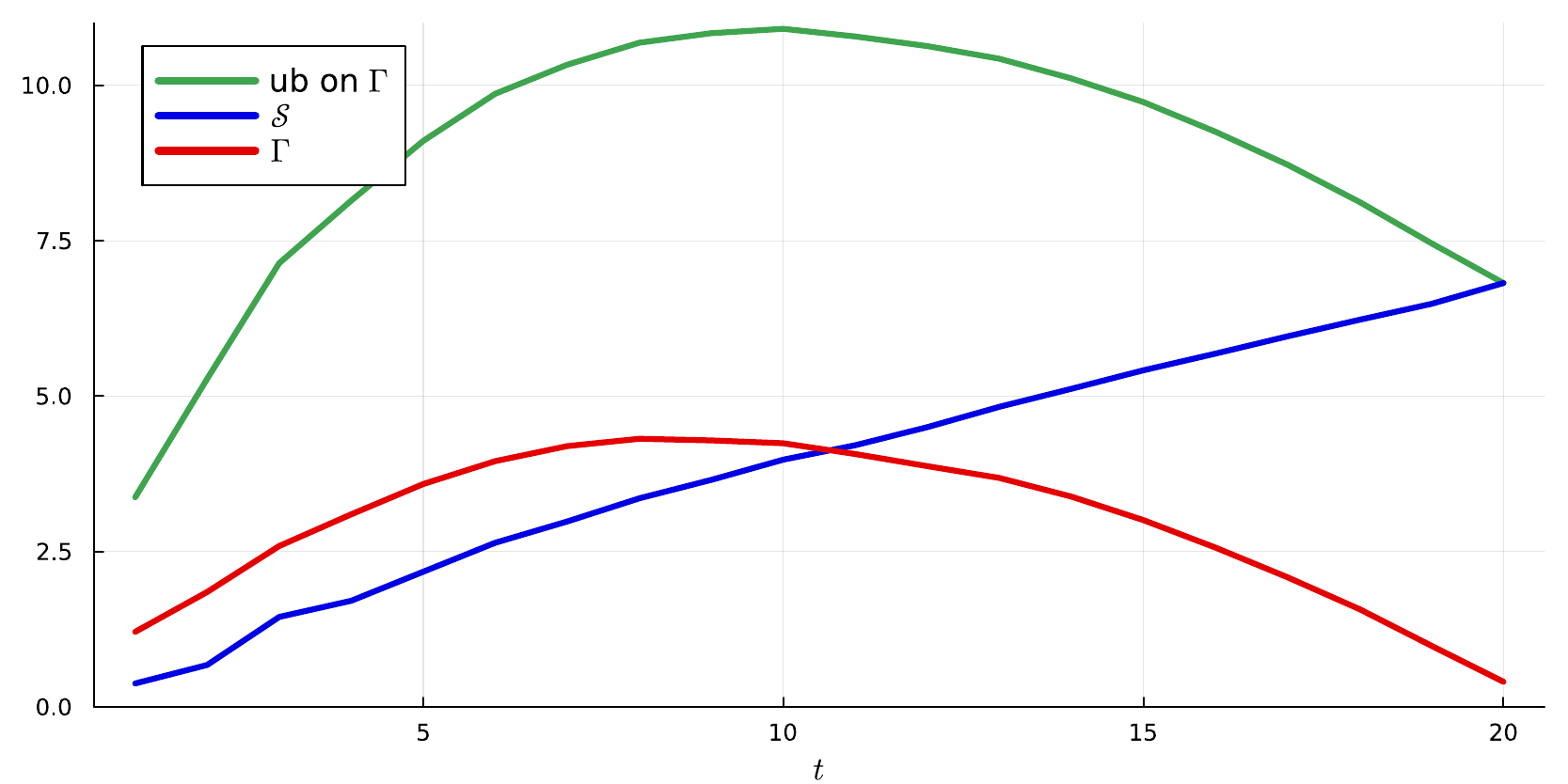}
    \caption{\ref{GMESP}, $s=20$}
    \label{fig:s20}
\end{subfigure}
\hfill
\begin{subfigure}{0.49\textwidth}
    \includegraphics[width=\textwidth]{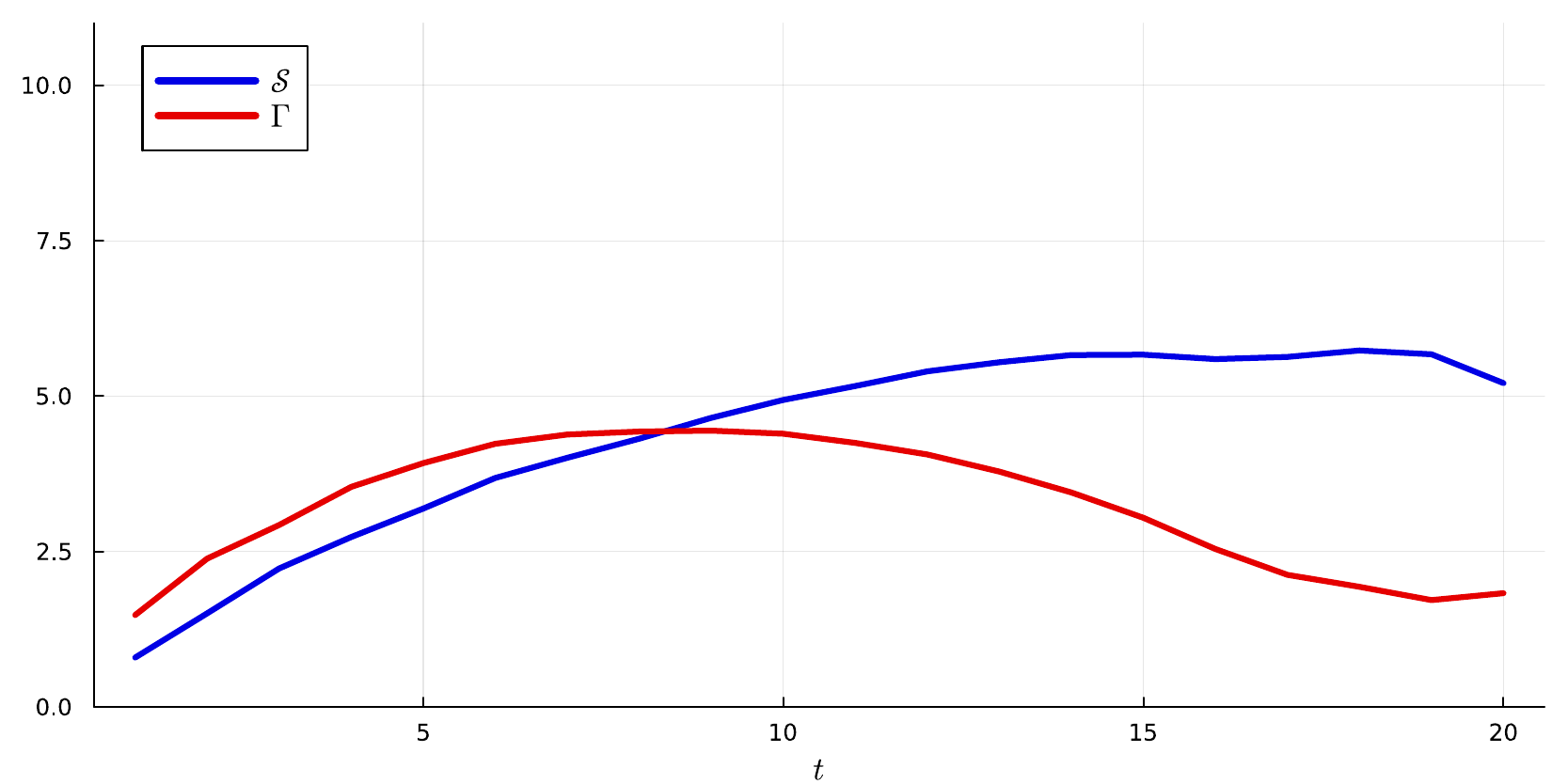}
    \caption{\ref{CGMESP}, $m=10$, $s=20$}
    \label{fig:s20c}
\end{subfigure}

\begin{subfigure}{0.49\textwidth}
    \includegraphics[width=\textwidth]{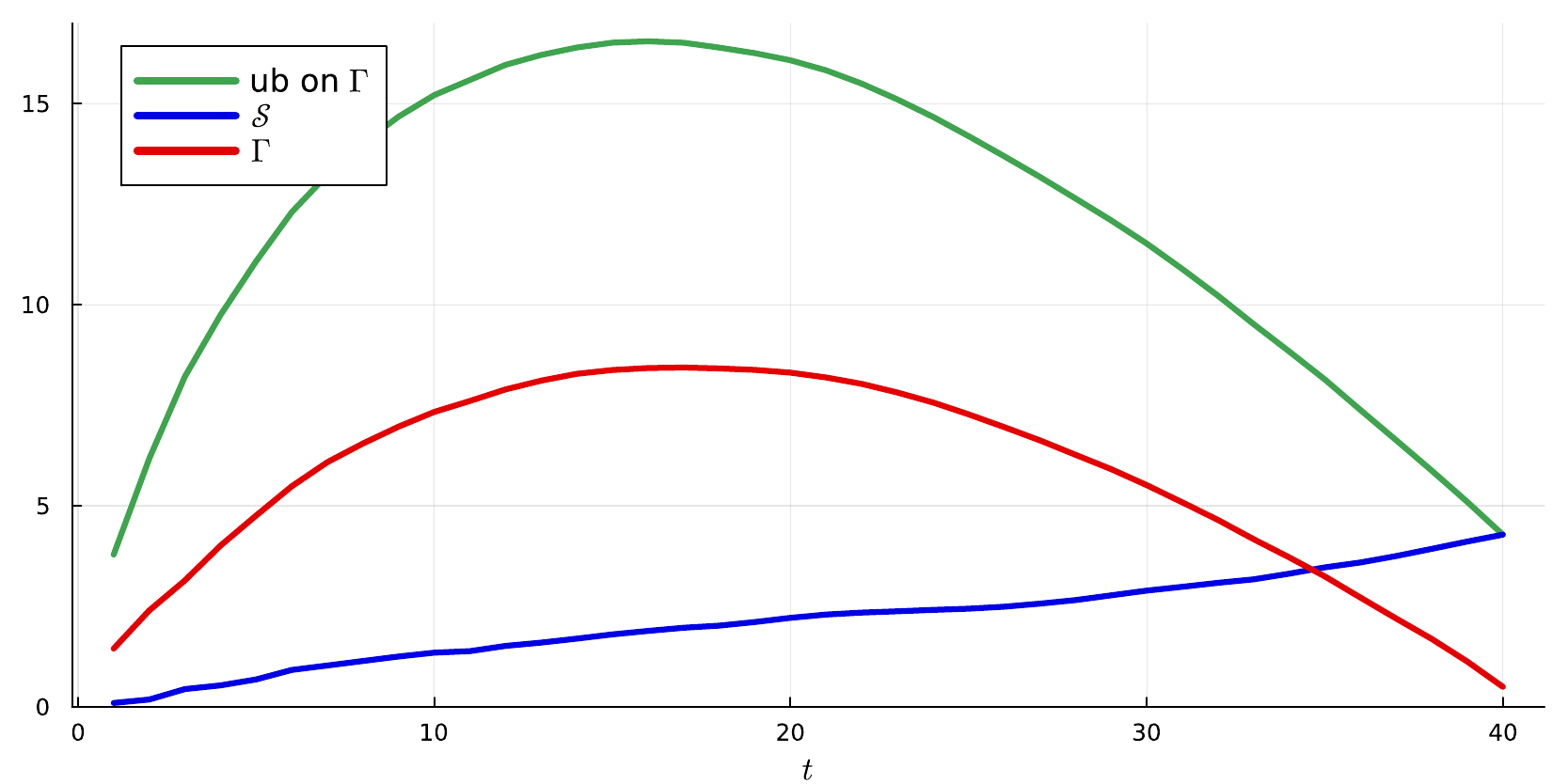}
    \caption{\ref{GMESP}, $s=40$}
    \label{fig:s40}
\end{subfigure}
\hfill
\begin{subfigure}{0.49\textwidth}
    \includegraphics[width=\textwidth]{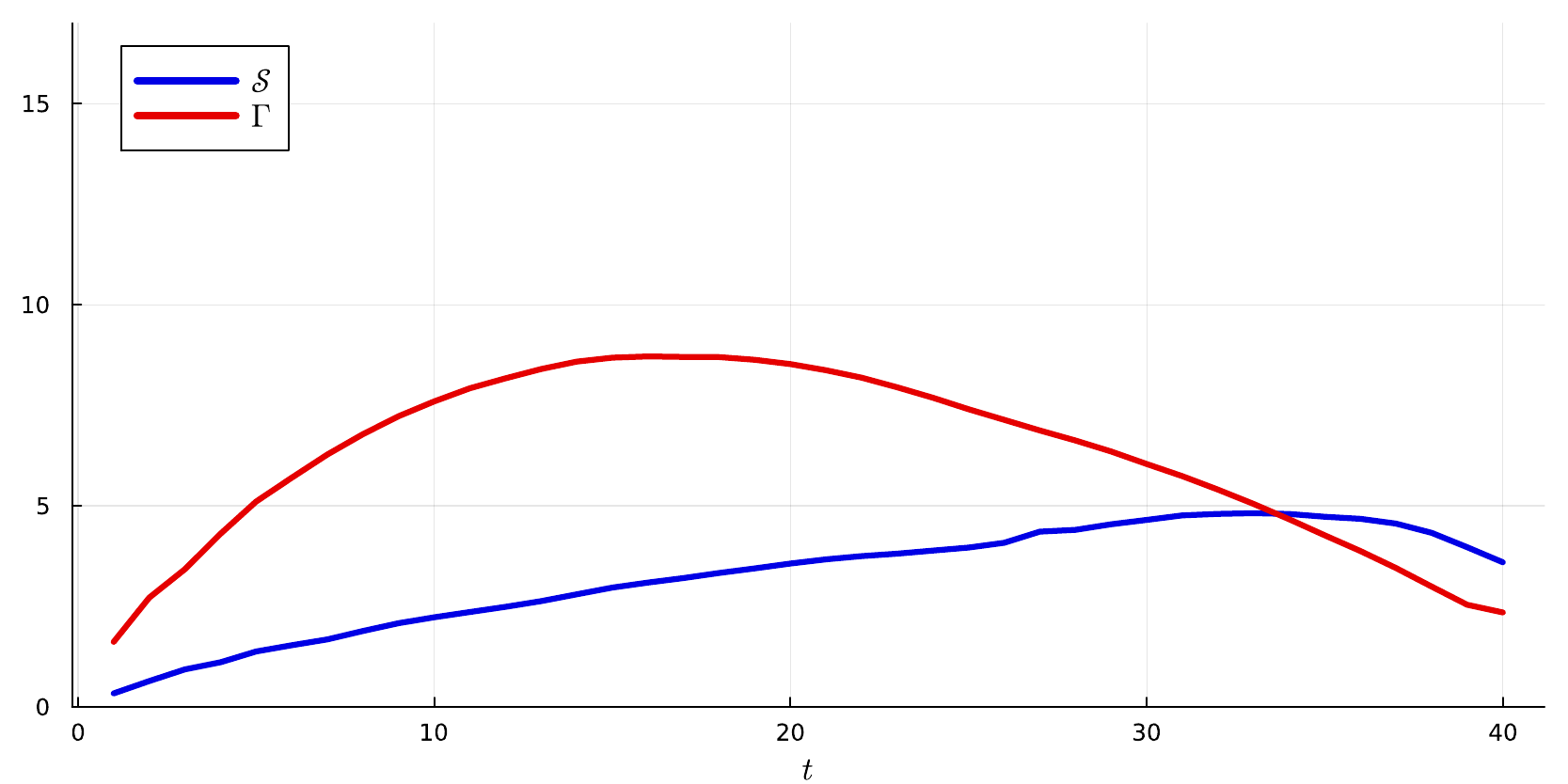}
    \caption{\ref{CGMESP}, $m=10$, $s=40$}
    \label{fig:s40c}
\end{subfigure}

\caption{Gaps for \ref{GMESP} and \ref{CGMESP}, varying $t$, with $s$ fixed ($n=63$)}
\label{fig:varying_t}
\end{figure}

 In each plot in  Figure \ref{fig:varying_t}, we see that after a certain value of $t$, the generalized factorization bound becomes stronger than the spectral bound. For \ref{GMESP}, the fraction of values of $t$ for which the generalized factorization bound is stronger is greater when $s$ is smaller; the same is true for \ref{CGMESP}. These observations are expected. We note that from  the formulae to calculate  the bounds,  the spectral bounds do not take $s$ into account, so it should become worse as $n-s$ becomes larger. On the other side,  the generalized factorization bound takes into account all  $s$ eigenvalues of the submatrix of $C$, so it should become worse as $s-t$ becomes larger.  When $t=s$, we see  that the generalized factorization bound dominates the spectral bound in all cases; once more, we note  that this is guaranteed when there are no side constraints (instances of MESP). Finally, we observe in Figure \ref{fig:varying_t}, that the bound on the gap for the generalized factorization bound for \ref{GMESP} is very loose. The plots confirm the analysis in Figures \ref{fig:varying_k} and \ref{fig:varying_kc}, showing now, for fixed $s$,  that the generalized factorization bound becomes more promising when the difference between $t$ and $s$ is small.

\subsection{Branch-and-bound}
 We assume familiarity with B\&B based on convex relaxation (see \cite{Muriqui2020}, for example). We note that, because \ref{DDGFact} is not an exact relaxation, we had to 
make some accommodations that are not  completely standard for B\&B.
More specifically, we have modified the \texttt{Juniper} code so that when we get an \emph{integer optimum} solution $\bar{x}$ for a relaxed  subproblem: 

\begin{enumerate}
\item[(i)] 
we evaluate the true objective value of $\bar{x}$ in  \ref{CGMESP}, to determine if we should increase the lower bound;
\item[(ii)] if the objective 
value of the relaxation of the subproblem being handled is greater than the lower bound, then we have to create child subproblems.
\end{enumerate}

We  initialize the B\&B with the lower bound computed as described in \S\ref{sec:LB}. Branching is the standard fixing of a variable at 0 or 1 for the two child subproblems. We do not create any child subproblem
for which there would be a unique 0/1 solution satisfying $\mathbf{e}^\top x = s$.

In Table \ref{tab:BB}, we present statistics for the B\&B applied  to instances where $C$ is the  leading principal submatrix of order $32$ of our $63$-dimensional covariance matrix and $s$ varies from 2 to 31. 
For each $s$, we solve an instance  of \ref{GMESP} with $t\!=\!s\!-\!1$ and an instance of MESP ($t\!=\!s$). For all instances of \ref{GMESP} and MESP, the initial lower bounds computed were  optimal, so the B\&B worked on proving optimality by decreasing the upper bound.  In the first column of Table \ref{tab:BB}, we show $s$, and in the other columns, we show the following statistics for \ref{GMESP} and MESP: the initial gap given by the difference between the upper and lower bounds at the root node (root gap),   the number of convex relaxations solved (nodes), the number of nodes pruned by bound (tot prun),  the number of variables fixed at 0 (1)  by the procedure described in Theorem \ref{thm:fixFact} (var fix 0 (1)), and the elapsed time (in seconds) for the B\&B (B\&B time).  For \ref{GMESP} we  additionally show: 

the number of nodes pruned by bound where the relaxations had integer optimum solutions (int prun),
 the number of integer optimum solutions obtained when solving relaxations (tot int),
 the mean  (resp., standard deviation) of the difference between the relaxations optimum values and  the objective function values of \ref{GMESP}/MESP at integer optimum solutions of relaxations 
(rel avg (resp, rel std)).

We observe that the difficulty of the problem significantly increases when $t$ becomes smaller than $s$, as the upper bounds become weaker, confirming the analysis of Figure \ref{fig:varying_k}. Nevertheless, the largest root gap for \ref{GMESP} is about one, and we can solve all instances in less than three hours. 
The quality of the generalized factorization bound for these instances of \ref{GMESP} can be evaluated by the number of nodes pruned by bound in the B\&B. For the nine most difficult instances (that took more than 5,000 seconds to be solved), 19\% of the nodes were pruned by bound on average.   Moreover, we notice that our matrix $C$ is full rank, therefore, the rank of the matrix $F(x)$ in the objective function of our relaxation is at least $s$. This means that the objective value of the relaxation is greater than the objective value of \ref{GMESP} even at integer solutions, and therefore, no node can be pruned by optimality, which would be  possible if the relaxation was exact. In fact,  when considering again only the nine most difficult instances of \ref{GMESP}, we see that, on average,  only on 5\% of the pruned nodes, the solution of the relaxation was integer. Moreover, for these nine instances, on average, only 5\% of the nodes where the relaxation had an integer optimum solution were pruned (in those nodes the  upper bound was not greater than the lower bound given by the heuristic solution). We can observe for these most difficult instances, an average difference of 0.73 between the relaxation objective value and the value of the objective function of \ref{GMESP}. We conclude that although the relaxation is not exact, the generalized factorization bound is strong enough to allow pruning a significant number of nodes in the B\&B.  Finally, we see that the bound is also strong enough to lead to an effective application of the variable-fixing procedure described in Theorem \ref{thm:fixFact}.

\begin{table}[ht!]
\resizebox{\textwidth}{!}{
\begin{tabular}{r|rrrrrrrrrr|rrrrrr}
                                & \multicolumn{10}{c|}{\ref{GMESP} ($t:=s-1$)}                                                                                                                                                                                                                                                                                                               & \multicolumn{6}{c}{MESP}                                                                                                                                             \\
\multicolumn{1}{c|}{\textit{s}} & \multicolumn{1}{c}{root} & \multicolumn{1}{c}{nodes} & \multicolumn{1}{c}{tot} & \multicolumn{1}{c}{int} & \multicolumn{1}{c}{tot} & \multicolumn{1}{c}{rel} & \multicolumn{1}{c}{rel} & \multicolumn{1}{c}{var}   & \multicolumn{1}{c}{var}   & \multicolumn{1}{c|}{B\&B} & \multicolumn{1}{c}{root} & \multicolumn{1}{c}{nodes} & \multicolumn{1}{c}{tot} & \multicolumn{1}{c}{var}   & \multicolumn{1}{c}{var}   & \multicolumn{1}{c}{B\&B} \\
                                & \multicolumn{1}{c}{gap}                    & \multicolumn{1}{c}{}      & \multicolumn{1}{c}{prun}                      & \multicolumn{1}{c}{prun}                          & \multicolumn{1}{c}{int}                     & \multicolumn{1}{c}{avg}                          & \multicolumn{1}{c}{std}                          & \multicolumn{1}{c}{fix 0} & \multicolumn{1}{c}{fix 1} & \multicolumn{1}{c|}{time} & \multicolumn{1}{c}{gap}  & \multicolumn{1}{c}{}      & \multicolumn{1}{c}{prun}  & \multicolumn{1}{c}{fix 0} & \multicolumn{1}{c}{fix 1} & \multicolumn{1}{c}{time} \\ \hline
2                               & 0.36                                                                & 186                       & 1                          & 1                              & 179                     & 0.44                        & 0.08                        & 75                        & 6                         & 1.44                      & 0.02                    & 3                         & 2                          & 0                         & 0                         & 0.04                     \\
3                               & 0.48                                                                & 557                       & 14                         & 9                              & 501                     & 0.55                        & 0.07                        & 269                       & 24                        & 4.22                      & 0.06                    & 10                        & 5                          & 52                        & 0                         & 0.08                     \\
4                               & 0.59                                                                & 2037                      & 89                         & 42                             & 1554                    & 0.61                        & 0.07                        & 1022                      & 109                       & 15.45                     & 0.12                     & 20                        & 8                          & 55                        & 0                         & 0.19                     \\
5                               & 0.70                                                                & 7404                      & 579                        & 187                            & 5070                    & 0.64                        & 0.07                        & 2724                      & 483                       & 42.28                     & 0.13                    & 33                        & 13                         & 58                        & 0                         & 0.29                     \\
6                               & 0.73                                                                & 20804                     & 2190                       & 569                            & 13251                   & 0.65                        & 0.07                        & 6668                      & 1657                      & 129.11                    & 0.14                    & 43                        & 16                         & 60                        & 0                         & 0.38                     \\
7                               & 0.76                                                                & 44592                     & 5500                       & 1194                           & 26958                   & 0.66                        & 0.08                        & 12582                     & 4141                      & 278.79                    & 0.14                    & 35                        & 13                         & 53                        & 0                         & 0.32                     \\
8                               & 0.77                                                                & 68682                     & 9697                       & 2056                           & 40021                   & 0.68                        & 0.08                        & 19272                     & 8011                      & 431.56                    & 0.13                    & 41                        & 16                         & 55                        & 0                         & 0.35                     \\
9                               & 0.82                                                                & 119525                    & 20711                      & 3648                           & 67307                   & 0.68                        & 0.08                        & 32902                     & 17728                     & 976.92                    & 0.12                     & 38                        & 15                         & 45                        & 0                         & 0.32                     \\
10                              & 0.82                                                                & 158998                    & 34012                      & 5048                           & 84404                   & 0.69                        & 0.08                        & 47229                     & 32943                     & 1112.10                    & 0.11                    & 37                        & 17                         & 49                        & 0                         & 0.29                     \\
11                              & 0.77                                                                & 147411                    & 35635                      & 4574                           & 72563                   & 0.69                        & 0.08                        & 50893                     & 46972                     & 1160.19                   & 0.06                    & 16                        & 8                          & 33                        & 0                         & 0.14                     \\
12                              & 0.72                                                                & 158170                    & 40735                      & 4522                           & 69865                   & 0.69                        & 0.08                        & 63096                     & 77637                     & 1491.04                   & 0.09                    & 46                        & 21                         & 42                        & 0                         & 0.41                     \\
13                              & 0.64                                                                & 142970                    & 36444                      & 3338                           & 53587                   & 0.70                        & 0.08                        & 67264                     & 100845                    & 1210.29                   & 0.17                    & 111                       & 49                         & 80                        & 0                         & 0.93                     \\
14                              & 0.68                                                                & 279210                    & 66511                      & 5059                           & 88512                   & 0.69                        & 0.08                        & 132296                    & 237550                    & 2821.80                    & 0.13                    & 196                       & 83                         & 131                       & 0                         & 1.72                     \\
15                              & 0.69                                                                & 386507                    & 87553                      & 5638                           & 106718                  & 0.68                        & 0.07                        & 176598                    & 358770                    & 5535.94                   & 0.13                    & 185                       & 81                         & 95                        & 0                         & 1.67                     \\
16                              & 0.69                                                                & 421261                    & 92967                      & 5301                           & 99083                   & 0.69                        & 0.07                        & 181361                    & 430153                    & 5178.73                   & 0.16                    & 323                       & 149                        & 136                       & 0                         & 2.83                     \\
17                              & 0.73                                                                 & 519556                    & 109265                     & 5012                           & 103521                  & 0.69                        & 0.06                        & 202088                    & 572510                    & 6399.18                   & 0.18                    & 305                       & 140                        & 107                       & 0                         & 2.77                     \\
18                              & 0.76                                                                & 544073                    & 110176                     & 4483                           & 98168                   & 0.70                        & 0.05                        & 190782                    & 625212                    & 6958.49                   & 0.25                    & 878                       & 391                        & 302                       & 0                         & 7.64                     \\
19                              & 0.83                                                                & 713393                    & 135220                     & 5655                           & 122085                  & 0.72                        & 0.04                        & 212475                    & 787816                    & 10159.16                  & 0.28                    & 1105                      & 489                        & 381                       & 0                         & 9.20                      \\
20                              & 0.86                                                                & 695101                    & 124354                     & 5504                           & 115999                  & 0.75                        & 0.03                        & 174689                    & 757696                    & 9751.58                   & 0.32                     & 1357                      & 586                        & 425                       & 0                         & 11.31                    \\
21                              & 0.90                                                                & 637108                    & 107189                     & 4280                           & 102540                  & 0.76                        & 0.03                        & 132122                    & 682763                    & 8584.51                   & 0.37                    & 2135                      & 850                        & 784                       & 0                         & 17.85                    \\
22                              & 0.95                                                                & 586069                    & 91694                      & 4994                           & 92685                   & 0.78                        & 0.03                        & 97344                     & 615472                    & 7617.47                   & 0.39                    & 2342                      & 855                        & 807                       & 0                         & 20.81                    \\
23                              & 0.99                                                                & 510390                    & 74438                      & 7159                           & 79882                   & 0.79                        & 0.03                        & 64390                     & 528729                    & 6330.14                   & 0.39                    & 1811                      & 660                        & 572                       & 0                         & 15.24                    \\
24                              & 1.01                                                                & 368953                    & 51277                      & 7806                           & 56157                   & 0.80                        & 0.03                        & 35163                     & 389241                    & 4149.70                    & 0.36                    & 1244                      & 463                        & 340                       & 0                         & 10.42                    \\
25                              & 1.01                                                                & 214937                    & 28168                      & 5440                           & 30889                   & 0.81                        & 0.02                        & 15377                     & 240272                    & 2149.56                   & 0.31                    & 632                       & 230                        & 130                       & 0                         & 5.12                     \\
26                              & 0.97                                                                & 98751                     & 12009                      & 2756                           & 13318                   & 0.81                        & 0.02                        & 5084                      & 121572                    & 921.15                    & 0.27                     & 463                       & 160                        & 65                        & 0                         & 3.74                     \\
27                              & 0.95                                                                & 43037                     & 4793                       & 1475                           & 5593                    & 0.82                        & 0.02                        & 1410                      & 58962                     & 389.61                    & 0.22                    & 241                       & 74                         & 30                        & 0                         & 1.99                     \\
28                              & 0.91                                                                & 13854                     & 1367                       & 511                            & 1667                    & 0.82                        & 0.02                        & 234                       & 22591                     & 128.65                    & 0.17                    & 130                       & 38                         & 10                        & 0                         & 1.15                     \\
29                              & 0.90                                                                  & 3403                      & 281                        & 134                            & 347                     & 0.82                        & 0.03                        & 20                        & 6547                      & 31.89                     & 0.13                    & 70                        & 16                         & 4                         & 0                         & 0.62                     \\
30                              & 0.87                                                                & 464                       & 27                         & 16                             & 35                      & 0.81                        & 0.05                        & 1                         & 949                       & 4.51                      & 0.08                    & 19                        & 4                          & 1                         & 0                         & 0.17                     \\
31                              & 0.83                                                                & 32                        & 1                          & 1                              & 1                       & 0.74                        & -                            & 0                         & 30                        & 0.32                      & 0.02                    & 2                         & 1                          & 0                         & 0                         & 0.02                    
\end{tabular}
}
\caption{Results for B\&B with variable fixing: \ref{GMESP}/MESP}
\label{tab:BB}
\end{table}

In Table \ref{tab:BB_constrained}, we present results for the experiments with \ref{CGMESP} and CMESP, where we added ten side constraints to the instances considered in Table \ref{tab:BB}. Unlike  in the case of \ref{GMESP}, the heuristic did not find the optimal solution for all instances of \ref{CGMESP}. Therefore, in addition to the statistics presented in Table \ref{tab:BB}, we also present for \ref{CGMESP},  in columns 3 and 4 of Table \ref{tab:BB_constrained}, the difference between the values of  the optimum solution and the heuristic solution  (heur gap), and the number of times the lower bound is increased during the execution of B\&B (imp inc).  Considering again the nine most difficult instances of \ref{CGMESP} (that took more than 10 seconds to be solved), we see that, on average,  17\% of the nodes were pruned by bound, and  in 9\% of them, the solution of the relaxation was integer. For these instances, we observe that, on average, 32\% of the nodes where the relaxation had an integer optimum solution were pruned. These results show that the integer solutions found for \ref{CGMESP} were more effective in reducing the total number of nodes of the B\&B tree. Nevertheless, the average difference between the relaxation objective value and the value of the objective function of \ref{CGMESP} at the integer solutions found for theses instances is 0.72, approximately the same that we had for \ref{GMESP}. Furthermore, we notice again that the difficulty of the problem increases when  $t$ becomes smaller than $s$. The time to solve \ref{CGMESP} is greater than the time to solve CMESP for most of the instances.  While the observations about the results for \ref{CGMESP} are similar to the observations made about the results for \ref{GMESP}, we see that the addition of side constraints made the instances much easier. The maximum time to solve an instance was reduced from about 10,000 seconds  to less than 600 seconds  when constraints were added, regardless of the fact that  the root gap is significantly larger for \ref{CGMESP} than for \ref{GMESP}. Although some of the increase in the root gap can be attributed to the gap between the optimal solution and the heuristic lower bound, in most cases it is actually caused by a larger difference between the upper bound at the root node and the optimum value. Despite this fact, the instances can be solved much faster. Finally, it is interesting to note that adding side constraints makes the difficulty of the instances harder to predict. Contrary to  what we see in Table \ref{tab:BB}, we see that times can increase and decrease  as the value of $s$ increases. 

\begin{table}[ht!]
\resizebox{\textwidth}{!}{
\begin{tabular}{r|rrrrrrrrrrrr|rrrrrr}
                                & \multicolumn{12}{c|}{\ref{CGMESP} ($t:=s-1$)}                                                                                                                                                                                                                                                                                                               & \multicolumn{6}{c}{CMESP}                                                                                                                                             \\
\multicolumn{1}{c|}{\textit{s}} & \multicolumn{1}{c}{root} & \multicolumn{1}{c}{heur} & \multicolumn{1}{c}{imp} & \multicolumn{1}{c}{nodes} & \multicolumn{1}{c}{tot} & \multicolumn{1}{c}{int} & \multicolumn{1}{c}{tot} & \multicolumn{1}{c}{rel} & \multicolumn{1}{c}{rel} & \multicolumn{1}{c}{var}   & \multicolumn{1}{c}{var}   & \multicolumn{1}{c|}{B\&B} & \multicolumn{1}{c}{root} & \multicolumn{1}{c}{nodes} & \multicolumn{1}{c}{tot} & \multicolumn{1}{c}{var}   & \multicolumn{1}{c}{var}   & \multicolumn{1}{c}{B\&B} \\
                                & \multicolumn{1}{c}{gap}  & \multicolumn{1}{c}{gap}                      & \multicolumn{1}{c}{inc}                     & \multicolumn{1}{c}{}      & \multicolumn{1}{c}{prun}                      & \multicolumn{1}{c}{prun}                          & \multicolumn{1}{c}{int}                     & \multicolumn{1}{c}{avg}                          & \multicolumn{1}{c}{std}                          & \multicolumn{1}{c}{fix 0} & \multicolumn{1}{c}{fix 1} & \multicolumn{1}{c|}{time} & \multicolumn{1}{c}{gap}  & \multicolumn{1}{c}{}      & \multicolumn{1}{c}{prun}  & \multicolumn{1}{c}{fix 0} & \multicolumn{1}{c}{fix 1} & \multicolumn{1}{c}{time} \\ \hline
2                               & 0.93                    & 0.37                    & 1                        & 121                       & 2                          & 1                              & 8                       & 0.43                        & 0.06                        & 149                       & 0                         & 6.80                      & 0.85                    & 23                        & 3                          & 63                        & 0                         & 5.63                     \\
3                               & 1.18                    & -                        & 0                        & 27                        & 0                          & 0                              & 3                       & 0.60                        & 0.01                        & 51                        & 0                         & 0.49                      & 0.88                    & 19                        & 2                          & 36                        & 0                         & 0.34                     \\
4                               & 0.99                    & -                        & 0                        & 35                        & 0                          & 0                              & 2                       & 0.60                        & 0.00                        & 93                        & 2                         & 0.59                      & 0.63                    & 17                        & 4                          & 63                        & 2                         & 0.25                     \\
5                               & 1.08                    & -                        & 0                        & 741                       & 76                         & 18                             & 43                      & 0.60                        & 0.08                        & 1213                      & 26                        & 10.85                     & 0.72                    & 144                       & 36                         & 246                       & 3                         & 2.25                     \\
6                               & 1.24                    & 0.18                    & 1                        & 419                       & 42                         & 5                              & 21                      & 0.62                        & 0.07                        & 790                       & 32                        & 6.49                      & 0.76                    & 97                        & 24                         & 169                       & 2                         & 1.45                     \\
7                               & 1.69                    & 0.03                    & 1                        & 107                       & 1                          & 1                              & 8                       & 0.67                        & 0.04                        & 121                       & 10                        & 1.98                      & 1.36                    & 57                        & 4                          & 92                        & 5                         & 1.04                     \\
8                               & 2.12                    & -                        & 0                        & 122                       & 0                          & 0                              & 1                       & 0.58                        & -                            & 74                        & 5                         & 2.26                      & 2.02                    & 105                       & 0                          & 108                       & 8                         & 1.86                     \\
9                               & 1.15                    & -                        & 0                        & 625                       & 39                         & 10                             & 70                      & 0.65                        & 0.09                        & 825                       & 160                       & 9.27                      & 0.69                    & 165                       & 41                         & 276                       & 66                        & 2.59                     \\
10                              & 1.21                    & -                        & 0                        & 325                       & 11                         & 0                              & 44                      & 0.69                        & 0.03                        & 372                       & 96                        & 3.59                      & 0.64                    & 84                        & 13                         & 85                        & 15                        & 1.43                     \\
11                              & 1.65                    & -                        & 0                        & 1046                      & 121                        & 6                              & 34                      & 0.68                        & 0.03                        & 1163                      & 452                       & 11.85                     & 1.53                    & 527                       & 127                        & 638                       & 262                       & 8.30                     \\
12                              & 2.75                    & 0.22                    & 2                        & 195                       & 0                          & 0                              & 6                       & 0.74                        & 0.02                        & 142                       & 38                        & 2.49                      & 2.56                    & 191                       & 2                          & 153                       & 83                        & 3.20                     \\
13                              & 1.77                    & 0.05                    & 1                        & 2904                      & 396                        & 22                             & 106                     & 0.67                        & 0.05                        & 2733                      & 1227                      & 43.59                     & 1.60                    & 1061                      & 319                        & 1432                      & 635                       & 17.42                    \\
14                              & 1.71                    & -                        & 0                        & 47                        & 0                          & 0                              & 3                       & 0.67                        & 0.00                        & 41                        & 23                        & 0.77                      & 1.45                    & 39                        & 0                          & 40                        & 33                        & 0.72                     \\
15                              & 1.14                    & -                        & 0                        & 1653                      & 392                        & 13                             & 58                      & 0.68                        & 0.03                        & 1266                      & 1022                      & 20.46                     & 0.75                    & 343                       & 117                        & 310                       & 225                       & 5.56                     \\
16                              & 1.36                    & 0.18                    & 1                        & 11198                     & 2592                       & 208                            & 907                     & 0.70                        & 0.04                        & 7157                      & 6339                      & 132.42                    & 1.04                    & 1898                      & 747                        & 2019                      & 1416                      & 32.60                    \\
17                              & 1.97                    & -                        & 0                        & 876                       & 75                         & 2                              & 9                       & 0.78                        & 0.03                        & 552                       & 668                       & 11.24                     & 1.53                    & 415                       & 79                         & 343                       & 366                       & 7.23                     \\
18                              & 1.68                    & 0.31                    & 1                        & 29149                     & 7943                       & 448                            & 1175                    & 0.74                        & 0.04                        & 13686                     & 18518                     & 353.74                    & 1.29                    & 3383                      & 1548                       & 3046                      & 3524                      & 59.50                    \\
19                              & 1.64                    & 0.32                    & 2                        & 3664                      & 557                        & 72                             & 252                     & 0.73                        & 0.04                        & 2035                      & 2806                      & 48.31                     & 1.16                    & 1030                      & 353                        & 938                       & 1175                      & 16.49                    \\
20                              & 2.55                    & -                        & 0                        & 209                       & 1                          & 0                              & 4                       & 0.76                        & 0.00                        & 75                        & 215                       & 2.99                      & 2.21                    & 213                       & 3                          & 75                        & 195                       & 3.77                     \\
21                              & 1.47                    & 0.04                    & 1                        & 47344                     & 10268                      & 820                            & 2043                    & 0.79                        & 0.03                        & 10281                     & 36017                     & 561.97                    & 0.96                    & 4812                      & 1846                       & 1197                      & 4043                      & 61.19                    \\
22                              & 2.91                    & 0.08                    & 1                        & 660                       & 1                          & 0                              & 4                       & 0.79                        & 0.05                        & 100                       & 360                       & 8.70                      & 2.45                    & 526                       & 8                          & 122                       & 473                       & 7.75                     \\
23                              & 1.45                    & -                        & 0                        & 135                       & 8                          & 0                              & 6                       & 0.67                        & 0.01                        & 22                        & 143                       & 1.64                      & 1.03                    & 69                        & 10                         & 21                        & 110                       & 1.18                     \\
24                              & 1.38                    & -                        & 0                        & 1677                      & 163                        & 14                             & 27                      & 0.75                        & 0.06                        & 223                       & 1959                      & 22.66                     & 0.79                    & 434                       & 114                        & 84                        & 572                       & 7.54                     \\
25                              & 1.68                    & -                        & 0                        & 199                       & 1                          & 0                              & 1                       & 0.72                        & -                            & 24                        & 201                       & 3.04                      & 1.23                    & 158                       & 8                          & 17                        & 193                       & 3.32                     \\
26                              & 3.62                    & -                        & 0                        & 55                        & 0                          & 0                              & 4                       & 0.72                        & 0.00                        & 1                         & 35                        & 0.79                      & 3.18                    & 31                        & 0                          & 1                         & 30                        & 0.53                     \\
27                              & 3.47                    & 0.97                    & 3                        & 164                       & 4                          & 1                              & 4                       & 0.73                        & 0.01                        & 5                         & 205                       & 2.23                      & 2.87                    & 131                       & 9                          & 3                         & 213                       & 2.50                     \\
28                              & 2.66                    & -                        & 0                        & 34                        & 0                          & 0                              & 1                       & 0.75                        & -                            & 1                         & 54                        & 0.55                      & 2.16                    & 32                        & 1                          & 0                         & 49                        & 0.79                     \\
29                              & 1.09                    & -                        & 0                        & 78                        & 5                          & 0                              & 1                       & 0.73                        & -                            & 1                         & 192                       & 1.17                      & 0.40                    & 39                        & 13                         & 1                         & 127                       & 0.83                     \\
30                              & 1.23                    & -                        & 0                        & 4                         & 0                          & 0                              & 0                       & -                            & -                            & 0                         & 10                        & 0.07                      & 0.66                    & 4                         & 0                          & 0                         & 22                        & 0.16                     \\
31                              & 0.74                    & -                        & 0                        & 2                         & 0                          & 0                              & 0                       & -                            & -                            & 0                         & 0                         & 0.03                      & 0.00                    & 1                         & 0                          & 0                         & 0                         & 0.02                    
\end{tabular}
}
\caption{Results for  B\&B with variable fixing: \ref{CGMESP}/CMESP}
\label{tab:BB_constrained}
\end{table}

\section{Outlook}\label{sec:out} We are left with some clear challenges. 
A key one is to obtain better upper bounds when $s-t$ is large  (which better fits the motivating PCA selection problem from \S1), in hopes of exactly solving \ref{GMESP} 
instances by B\&B in such cases.
In connection with this, 
we would like a bound that provably dominates the spectral bound (when $t<s$), improving on
what we established with
Theorem 
\ref{thm:almost_dominates} and observed in Figure \ref{fig:varying_k}. Also, we would like to 
improve Theorem \ref{thm:almost_dominates} for the constrained case.

\medskip\noindent
{\bf Acknowledgments.}
\thanks{G. Ponte was supported in part by CNPq GM-GD scholarship 161501/2022-2. 
M. Fampa was supported in part by CNPq grant 307167/ 2022-4.
J. Lee was supported in part by AFOSR grant FA9550-22-1-0172.
}

\bibliography{fact_paper} 

\newcommand{\etalchar}[1]{$^{#1}$}
\begin{thebibliography}{AFLW99}

\bibitem[AFLW96]{AFLW_IPCO}
Kurt~M. Anstreicher, Marcia Fampa, Jon Lee, and Joy Williams.
\newblock Continuous relaxations for constrained maximum-entropy sampling.
\newblock In W.H. Cunningham, S.T. McCormick, and M.~Queyranne, editors, {\em Proceedings of IPCO 1996}, volume 1084 of {\em Lecture Notes in Computer Science}, pages 234--248. Springer, Berlin, 1996.
\newblock \url{https://doi.org/10.1007/3-540-61310-2_18}.

\bibitem[AFLW99]{AFLW_Using}
Kurt~M. Anstreicher, Marcia Fampa, Jon Lee, and Joy Williams.
\newblock Using continuous nonlinear relaxations to solve constrained maximum-entropy sampling problems.
\newblock {\em Mathematical Programming, Series A}, 85(2):221--240, 1999.
\newblock \url{https://doi.org/10.1007/s101070050055}.

\bibitem[AL04]{AnstreicherLee_Masked}
Kurt~M. Anstreicher and Jon Lee.
\newblock A masked spectral bound for maximum-entropy sampling.
\newblock In A.~Di~Bucchianico, H.~Läuter, and H.P. Wynn, editors, {\em Proceedings of m{OD}a 7 --- {A}dvances in Model-Oriented Design and Analysis}, Contributions to Statistics, pages 1--12. Physica, Heidelberg, 2004.
\newblock \url{https://doi.org/10.1007/978-3-7908-2693-7_1}.

\bibitem[Ans18]{Anstreicher_BQP_entropy}
Kurt~M. Anstreicher.
\newblock {Maximum-entropy sampling and the Boolean quadric polytope}.
\newblock {\em Journal of Global Optimization}, 72(4):603--618, 2018.
\newblock \url{https://doi.org/10.1007/s10898-018-0662-x}.

\bibitem[Ans20]{Kurt_linx}
Kurt~M. Anstreicher.
\newblock Efficient solution of maximum-entropy sampling problems.
\newblock {\em Operations Research}, 68(6):1826--1835, 2020.
\newblock \url{https://doi.org/10.1287/opre.2019.1962}.

\bibitem[BL07]{BurerLee}
Samuel Burer and Jon Lee.
\newblock Solving maximum-entropy sampling problems using factored masks.
\newblock {\em Mathematical Programming, Series B}, 109(2-3):263--281, 2007.
\newblock \url{https://doi.org/10.1007/s10107-006-0024-1}.

\bibitem[CFL23]{ChenFampaLee_Fact}
Zhongzhu Chen, Marcia Fampa, and Jon Lee.
\newblock On computing with some convex relaxations for the maximum-entropy sampling problem.
\newblock {\em INFORMS Journal on Computing}, 35(2):368--385, 2023.
\newblock \url{https://doi.org/10.1287/ijoc.2022.1264}.

\bibitem[CFL25]{ChenFampaLeeGenScaling}
Zhongzhu Chen, Marcia Fampa, and Jon Lee.
\newblock Generalized scaling for the constrained maximum-entropy sampling problem.
\newblock {\em Mathematical Programming}, 212:177--216, 2025.
\newblock \url{https://doi.org/10.1007/s10107-024-02101-3}.

\bibitem[CFLL21]{Mixing}
Zhongzhu Chen, Marcia Fampa, Am\'elie Lambert, and Jon Lee.
\newblock Mixing convex-optimization bounds for maximum-entropy sampling.
\newblock {\em Mathematical Programming, Series B}, 188:539--568, 2021.
\newblock \url{https://doi.org/10.1007/s10107-020-01588-w}.

\bibitem[FL22]{FLBook}
Marcia Fampa and Jon Lee.
\newblock {\em Maximum-Entropy Sampling: Algorithms and Application}.
\newblock Springer, 2022.
\newblock \url{https://doi.org/10.1007/978-3-031-13078-6}.

\bibitem[GLSZ93]{Guttorp-Le-Sampson-Zidek1993}
Peter Guttorp, Nhu~D. Le, Paul~D. Sampson, and James~V. Zidek.
\newblock Using entropy in the redesign of an environmental monitoring network.
\newblock In G.P. Patil, C.R. Rao, and N.P. Ross, editors, {\em Multivariate Environmental Statistics}, volume~6, pages 175--202. North-Holland, 1993.
\newblock \url{https://marciafampa.com/pdf/GLSZ.pdf}.

\bibitem[HLW01]{HLW}
Alan Hoffman, Jon Lee, and Joy Williams.
\newblock New upper bounds for maximum-entropy sampling.
\newblock In A.C. Atkinson, P.~Hackl, and W.G. Müller, editors, {\em Proceedings of m{OD}a 6 --- Advances in Model-Oriented Design and Analysis ({P}uchberg/{S}chneeberg, 2001)}, Contributions in Statistics, pages 143--153. Physica, Heidelberg, 2001.
\newblock \url{https://doi.org/10.1007/978-3-642-57576-1_16}.

\bibitem[JC16]{PCA}
Ian Jolliffe and Jorge Cadima.
\newblock Principal component analysis: A review and recent developments.
\newblock {\em Philosophical Transactions of the Royal Society A: Mathematical, Physical and Engineering Sciences}, 374:20150202, 04 2016.
\newblock \url{https://doi.org/10.1098/rsta.2015.0202}.

\bibitem[KBS93]{byrd}
H.~Fayez Khalfan, R.~H. Byrd, and R.~B. Schnabel.
\newblock A theoretical and experimental study of the symmetric rank-one update.
\newblock {\em SIAM Journal on Optimization}, 3(1):1--24, 1993.
\newblock \url{https://doi.org/10.1137/0803001}.

\bibitem[KCHN18]{juniper}
Ole Kröger, Carleton Coffrin, Hassan Hijazi, and Harsha Nagarajan.
\newblock Juniper: An open-source nonlinear branch-and-bound solver in {J}ulia.
\newblock In W.J. van Hoeve, editor, {\em Proceedings of CPAIOR 2018}, pages 377--386, 2018.
\newblock \url{https://doi.org/10.1007/978-3-319-93031-2_27}.

\bibitem[KGL95]{gritz}
Victor Klee, Peter Gritzmann, and David Larman.
\newblock Largest $j$-simplices in $n$-polytopes.
\newblock {\em Discrete \& Computational Geometry}, 13(3--4):477--516, 1995.
\newblock \url{https://doi.org/10.1007/BF02574058}.

\bibitem[KLQ95]{KLQ}
Chun-Wa Ko, Jon Lee, and Maurice Queyranne.
\newblock An exact algorithm for maximum entropy sampling.
\newblock {\em Operations Research}, 43(4):684--691, 1995.
\newblock \url{https://doi.org/10.1287/opre.43.4.684}.

\bibitem[Lee24]{LeeLP}
Jon Lee.
\newblock {\em A First Course in Linear Optimization ({F}ourth Edition, v. 4.08)}.
\newblock Reex Press, 2013--24.
\newblock \url{https://github.com/jon77lee/JLee_LinearOptimizationBook}.

\bibitem[Lee98]{LeeConstrained}
Jon Lee.
\newblock Constrained maximum-entropy sampling.
\newblock {\em Operations Research}, 46(5):655--664, 1998.
\newblock \url{https://doi.org/10.1287/opre.46.5.655}.

\bibitem[LFL{\etalchar{+}}24]{li2022d}
Yongchun Li, Marcia Fampa, Jon Lee, Feng Qiu, Weijun Xie, and Rui Yao.
\newblock D-optimal data fusion: Exact and approximation algorithms.
\newblock {\em INFORMS Journal on Computing}, 36(1):97--120, 2024.
\newblock \url{https://doi.org/10.1287/ijoc.2022.0235}.

\bibitem[LL20]{LeeLind2020}
Jon Lee and Joy Lind.
\newblock Generalized maximum-entropy sampling.
\newblock {\em INFOR: Information Systems and Operations Research}, 58(2):168--181, 2020.
\newblock \url{https://doi.org/10.1080/03155986.2018.1533774}.

\bibitem[LO13]{nonsmoothBFGS}
Adrian~S. Lewis and Michael~L. Overton.
\newblock Nonsmooth optimization via quasi-{N}ewton methods.
\newblock {\em Mathematical Programming}, 141(1--2):135--163, 2013.
\newblock \url{ https://doi.org/10.1007/s10107-012-0514-2}.

\bibitem[LW03]{LeeWilliamsILP}
Jon Lee and Joy Williams.
\newblock A linear integer programming bound for maximum-entropy sampling.
\newblock {\em Mathematical Programming, Series B}, 94(2--3):247--256, 2003.
\newblock \url{https://doi.org/10.1007/s10107-002-0318-x}.

\bibitem[LX23]{Weijun}
Yongchun Li and Weijun Xie.
\newblock Best principal submatrix selection for the maximum entropy sampling problem: scalable algorithms and performance guarantees.
\newblock {\em Operations Research}, 72(2):493--513, 2023.
\newblock \url{https://doi.org/10.1287/opre.2023.2488}.

\bibitem[MFR20]{Muriqui2020}
Wendel Melo, Marcia Fampa, and Fernanda Raupp.
\newblock An overview of {MINLP} algorithms and their implementation in {Muriqui Optimizer}.
\newblock {\em Annals of Operations Research}, 286(1):217--241, 2020.
\newblock \url{https://doi.org/10.1007/s10479-018-2872-5}.

\bibitem[Nik15]{Nikolov}
Aleksandar Nikolov.
\newblock Randomized rounding for the largest simplex problem.
\newblock In R.~Rubinfeld, editor, {\em Proceedings of STOC 2015}, pages 861--870, 2015.
\newblock \url{https://doi.org/10.1145/2746539.2746628}.

\bibitem[PFL24]{SEA_proceedings}
Gabriel Ponte, Marcia Fampa, and Jon Lee.
\newblock Convex relaxation for the generalized maximum-entropy sampling problem.
\newblock In L.~Liberti, editor, {\em Proceedings of SEA 2024 (22nd International Symposium on Experimental Algorithms)}, volume 301 of {\em Leibniz International Proceedings in Informatics (LIPIcs)}, pages 25:1--25:14, 2024.
\newblock \url{https://doi.org/10.4230/LIPIcs.SEA.2024.25}.

\bibitem[PFL25]{PonteFampaLeeMPB}
Gabriel Ponte, Marcia Fampa, and Jon Lee.
\newblock Branch-and-bound for integer {D}-optimality with fast local search and variable-bound tightening.
\newblock {\em Mathematical Programming}, 2025.
\newblock \url{http://doi.org/10.1007/s10107-025-02196-2}.

\bibitem[Sah96]{Baron1996}
Nikolaos~V. Sahinidis.
\newblock {BARON:} {A} general purpose global optimization software package.
\newblock {\em Journal of Global Optimization}, 8(2):201--205, 1996.
\newblock \url{https://doi.org/10.1007/BF00138693}.

\bibitem[SC67]{Snedecor}
George~W. Snedecor and William~G. Cochran.
\newblock {\em Statistical Methods}.
\newblock Iowa State University Press, Ames, IA, sixth edition, 1967.

\bibitem[Sha48]{Shannon}
Claude~E. Shannon.
\newblock A mathematical theory of communication.
\newblock {\em The Bell System Technical Journal}, 27(3):379--423, 1948.
\newblock \url{https://doi.org/10.1002/j.1538-7305.1948.tb01338.x}.

\bibitem[Wil98]{WilliamsPhD}
Joy~Denise Williams.
\newblock {\em Spectral Bounds for Entropy Models}.
\newblock {Ph.D. thesis}, University of Kentucky, April 1998.
\newblock \url{https://saalck-uky.primo.exlibrisgroup.com/permalink/01SAA_UKY/15remem/alma9914832986802636}.

\bibitem[ZBLN97]{LBFGSB}
Ciyou Zhu, Richard~H. Byrd, Peihuang Lu, and Jorge Nocedal.
\newblock Algorithm 778: {L-BFGS-B:} {F}ortran subroutines for large-scale bound-constrained optimization.
\newblock {\em ACM Transactions on Mathematical Software}, 23(4):550--560, 1997.
\newblock \url{https://doi.org/10.1145/279232.279236}.

\end{thebibliography}

\end{document}